\documentclass{article}
\usepackage[T1]{fontenc}
\usepackage{geometry}
\usepackage{mathrsfs}
\usepackage{mathtools}
\usepackage{amsmath}
\usepackage{amsthm}
\usepackage{amssymb}
\usepackage{graphicx}
\usepackage{xargs}[2008/03/08]

\makeatletter
\numberwithin{figure}{section}
\numberwithin{equation}{section}
\theoremstyle{plain}
\newtheorem{thm}{\protect\theoremname}[section]
\theoremstyle{definition}
\newtheorem{defn}[thm]{\protect\definitionname}
\theoremstyle{plain}
\newtheorem{assumption}[thm]{\protect\assumptionname}
\theoremstyle{definition}
\newtheorem{example}[thm]{\protect\examplename}
\theoremstyle{plain}
\newtheorem{prop}[thm]{\protect\propositionname}
\theoremstyle{remark}
\newtheorem{rem}[thm]{\protect\remarkname}
\theoremstyle{plain}
\newtheorem{lem}[thm]{\protect\lemmaname}
\theoremstyle{plain}
\newtheorem{cor}[thm]{\protect\corollaryname}

\makeatother

\providecommand{\assumptionname}{Assumption}
\providecommand{\corollaryname}{Corollary}
\providecommand{\definitionname}{Definition}
\providecommand{\examplename}{Example}
\providecommand{\lemmaname}{Lemma}
\providecommand{\propositionname}{Proposition}
\providecommand{\remarkname}{Remark}
\providecommand{\theoremname}{Theorem}

\begin{document}
\title{A dynamic analytic method for risk-aware controlled martingale problems}
\author{Jukka Isoh\"at\"al\"a and William B. Haskell}
\maketitle
\begin{abstract}
We present a new, tractable method for solving and analyzing
risk-aware control problems over finite and infinite, discounted
time-horizons where the dynamics of the controlled process are
described as a martingale problem. Supposing general Polish state and
action spaces, and using generalized, relaxed controls, we state a
risk-aware dynamic optimal control problem of minimizing risk of costs
described by a generic risk function. We then construct an alternative
formulation that takes the form of a nonlinear programming problem,
constrained by the dynamic, {i.e.} time-dependent, and linear
Kolmogorov forward equation describing the distribution of the state
and accumulated costs. We show that the formulations are equivalent,
and that the optimal control process can be taken to be Markov in the
controlled process state, running costs, and time. We further prove
that under additional conditions, the optimal value is attained. An
example numeric problem is presented and solved.

\vspace{0.1cm}

\sloppypar
\vspace{0.1cm}
\end{abstract}

\global\long\def\D{\mathrm{d}}%
\global\long\def\E{\mathrm{e}}%
\global\long\def\RR{\mathbb{R}}%
\global\long\def\XX{\mathbb{X}}%
\global\long\def\YY{\mathbb{Y}}%
\global\long\def\deq{\coloneqq}%
\global\long\def\PpP{\mathfrak{P}}%
\global\long\def\dom{\mathop{\mathrm{dom}}}%
\global\long\def\NN{\mathbb{N}}%
\global\long\def\dD{\mathscr{D}}%
\global\long\def\Pp{\mathcal{P}}%
\global\long\def\Vv{\mathcal{V}}%
\global\long\def\pP{\mathscr{P}}%
\global\long\def\Cc{\mathcal{C}}%
\global\long\def\xX{\mathscr{X}}%
\global\long\def\yY{\mathscr{Y}}%
\global\long\def\Xx{\mathcal{X}}%
\global\long\def\bB{\mathscr{B}}%
\global\long\def\Aa{\mathcal{A}}%
\global\long\def\ordto{\stackrel{o}{\longrightarrow}}%
\global\long\def\cC{\mathscr{C}}%
\global\long\def\precfsd{\preceq_{\mathrm{fsd}}}%
\global\long\def\argmin{\mathop{\mathrm{arg\,min}}}%
\global\long\def\eE{\mathscr{E}}%
\global\long\def\aA{\mathscr{A}}%
\global\long\def\Nn{\mathcal{N}}%
\global\long\def\NnN{\mathfrak{N}}%
\global\long\def\Gg{\mathcal{G}}%
\global\long\def\TT{\mathrm{\mathbb{T}}}%
\global\long\def\EE{\mathrm{\mathbb{E}}}%
\global\long\def\MmM{\mathfrak{M}}%
\global\long\def\LlL{\mathfrak{L}}%
\global\long\def\Mm{\mathcal{M}}%
\global\long\def\CcC{\mathfrak{C}}%
\global\long\def\supp{\mathop{\mathrm{supp}}}%
\global\long\def\rR{\mathscr{R}}%
\global\long\def\PP{\mathbb{P}}%
\global\long\def\CVaR{\mathrm{CVaR}}%
\global\long\def\XX{\mathbb{X}}%
\global\long\def\lL{\mathscr{L}}%
\global\long\def\Ll{\mathcal{L}}%
\global\long\def\AA{\mathbb{A}}%
\global\long\def\eqd{\eqqcolon}%
\global\long\def\UU{\mathbb{U}}%
\global\long\def\TtT{\mathfrak{T}}%
\global\long\def\ttT{\mathfrak{t}}%
\global\long\def\FfF{\mathfrak{F}}%
\global\long\def\Leb{\mathrm{Leb}}%
\global\long\def\Lebae{\text{Leb-a.e.}}%
\global\long\def\compl{\mathrm{c}}%
\global\long\def\ddD{\mathfrak{d}}%
\global\long\def\KK{\mathbb{K}}%
\global\long\def\WW{\mathbb{W}}%
\global\long\def\jjJ{\mathfrak{j}}%
\global\long\def\DD{\mathbb{D}}%
\global\long\def\DdD{\mathfrak{D}}%
\global\long\def\wwW{\mathfrak{w}}%
\global\long\def\HH{\mathbb{H}}%
\global\long\def\Ff{\mathcal{F}}%
\global\long\def\RrR{\mathfrak{R}}%
\global\long\def\AaA{\mathfrak{A}}%
\global\long\def\pPM{\pP_{\Ll}}%
\global\long\def\pPC{\pP_{\Pp}}%
\global\long\def\spn{\mathop{\text{span}}}%
\global\long\def\bplim{\mathop{\text{bp-lim}}}%
\global\long\def\esssup{\mathop{\mathrm{ess}\,\mathrm{sup}}}%
\global\long\def\indic{\mathbb{I}}%
\global\long\def\Hh{\mathcal{H}}%
\global\long\def\interior{\mathop{\mathrm{int}}}%
\global\long\def\Oo{\mathcal{O}}%
\global\long\def\QQ{\mathbb{Q}}%
\newcommandx\aforall[1][usedefault, addprefix=\global, 1=\PP]{#1\text{-}\forall}%
\newcommandx\pesssup[1][usedefault, addprefix=\global, 1=\PP]{\mathop{#1\text{-}\mathrm{ess\,sup}}}%
\global\long\def\sgn{\mathop{\mathrm{sgn}}}%
\global\long\def\VV{\mathbb{V}}%
\global\long\def\Lip{\mathrm{Lip}}%
\global\long\def\tounif{\stackrel{\text{unif.}}{\longrightarrow}}%
\global\long\def\Tounif{\stackrel{\text{unif.}}{\Longrightarrow}}%
\global\long\def\ulim{\mathop{\mathrm{unif\,lim}}}%
\global\long\def\ulimsup{\mathop{\mathrm{unif\,limsup}}}%
\global\long\def\uliminf{\mathop{\mathrm{unif\,liminf}}}%
\global\long\def\gG{\mathscr{G}}%
\global\long\def\HH{\mathbb{H}}%
\global\long\def\warrow{\stackrel{w}{\longrightarrow}}%
\global\long\def\ccarrow{\stackrel{ucc}{\longrightarrow}}%
\global\long\def\Jj{\mathcal{J}}%
\global\long\def\id{\mathrm{id}}%
\global\long\def\trace{\mathop{\mathrm{tr}}}%

\section{Introduction}
We consider the risk-aware optimization of controlled stochastic
processes over a finite $\TT=[0,T]$, $T>0$, or infinite time-horizon
$\TT=[0,\infty)$ on general Polish state and action spaces $\XX$ and
  $\AA$. That is, for a filtered probability space $(\Omega, \Sigma,
  \Ff, \PP)$, we solve
  \begin{subequations}
    \label{eq:intro_risk}
  \begin{gather}
      \inf_{a}\rho\left(\int_{0}^{\infty}\E^{-\alpha t}c(x_{t},a_{t},t)\,\D t\right)\quad (\TT = [0,\infty)) \quad \text{or}
        \\
        \inf_{a}\rho\left(\int_{0}^{T}c(x_{t},a_{t},t)\,\D t + v(x_T) \right)\quad (\TT = [0,T]),
  \end{gather}
  \end{subequations}
where $\rho : \Ll(\Omega; \RR) \to \RR \cup \{\infty\}$ is a risk
function, $\alpha>0$ is the discount rate, $c : \XX \times \AA \times
\TT \to \RR$, and $v : \XX \to \RR$ are given cost rate and terminal
cost functions, and the infima run over sets of admissible of control
processes $a$ while $x$ are the controlled stochastic processes.
Here, the controlled processes shall be determined by the martingale
formulation, and we will consider generalized, relaxed controls.

The introduction of the risk function $\rho$ sets our control problem
apart from the classical, or \emph{risk-neutral} problem where $\rho$
is the expectation, denoted $\EE$. The expectation judges events of
high probability and low cost with the same standard as unlikely
events but high costs, and this may often be undesirable. Risks
matter, and our intuitions immediately weigh minor recurring
adversities differently from major catastrophes. The role of the risk
function is to describe the controllers preferences that may feature
{e.g.} tail-risk avoidance, loss aversion, or even risk-seeking
tendencies. Practical risk-management applications need proper models
of risk and risk preferences, and various risk functions such as the
conditional value-at-risk \cite{Rockafellar2000}, or classes of risk
functions \cite{Artzner99, Follmer02} have become an import tool in
risk modeling \cite{Pflug2007}.

The motivation of this work is constructing a tractable and readily
generalizable method to solving problems of the form of
Problem~(\ref{eq:intro_risk}). Approaches for risk-neutral dynamic
control do not trivially generalize to the risk-aware setup. The
standard approaches can be broadly binned into three categories:
(\emph{i}) dynamic programming, (\emph{ii}) probabilistic methods, and
(\emph{iii}) the convex analytic approach. All may be applicable to
risk-aware problems, but none are without notable issues. Dynamic
programming methods \cite{Lions83I,Lions83II,Soner06} form arguably
the most well-known and most frequently used apprach in the
risk-neutral case.  However, the derivation of the dynamic programming
equations hinges on the properties of the expectation that are not
always shared by risk functions. Overcoming this generally requires
considering \emph{dynamic risk measures} that impose nontrivial
additional structure on the risk functions \cite{Acciaio2011}.

\sloppypar
The second, probabilistic group of methods include most notably
various formulations of the stochastic Pontryagin's minimum principle,
see {e.g.} \cite{Peng1990, Zhou1991, Yong1999}. In this context,
solutions to the optimal control problem are found from stochastic
equations, hence the descriptor ``probabilistic.'' These methods have
been amenable to risk-aware problems, and risk-aware specific
approaches have been successfully developed, see
{e.g.} \cite{Peng2004_NonlinE, Peng2008, Peng2010}, though these too are
constrained to specific forms of dynamic risk measures. Recently, an
alternative formulation for generic (not necessarily dynamic) risk
functions was also found \cite{Isohatala2020_subm}. The solution of
the probabilistic formulations nonetheless involves the nontrivial
task of solving systems of forward-backward systems of stochastic
dynamic equations.

Convex analytic methods recast the dynamic control problem to a static
problem of optimizing over distributions, often called occupation
measures. In the risk-neutral case, this approach conventionally
yields linear programming problems, {e.g.} in the discounted infinite
time-horizon setup without explicit time-dependence,
\begin{align}
 & \begin{aligned}\inf_{\mu\geq0} & \int_{\XX\times\AA}c\left(x,a\right)\mu\left(\D x\times\D a\right)\\
\text{s.t. } & L(\mu) = \alpha \nu_0.
\end{aligned}
\label{eq:intro_LP}
\end{align}
Here, the measure $\mu$ represents the (discounted)
likelihood of the state-control pair visiting a given point in the
state-action space, and $L(\mu) = \alpha \nu_0$, the adjoint equation,
linear in $\mu$, encodes a set of constraints that determine the
occupation measure.  The measure $\nu_0$ is the initial distribution
for the controlled process. Proving the equivalence of
Problem~(\ref{eq:intro_LP}) and Problem~(\ref{eq:intro_risk}) with
$\rho = \EE$ requires showing that a solution of one of the problems
yields a solution to the other.

The convex analytic approach extends to the risk-aware case more
readily than dynamic programming methods, as risk functions can
evaluate risks from the cost distributions, and derivations of the
convex analytic problem do not heavily rely on the properties of the
expectation. 
In \cite{haskell15}, a state space augmentation scheme similar to that
of \cite{bauerle14} was used to derive a risk-aware convex analytic
formulation in discrete time. However, as convex analytic methods
construct the occupation measures from long-run, discounted visitation
frequencies, recovering the full cost distribution from the adjoint
equation $L(\mu) = \alpha \nu_0$ becomes technically awkward.

Here, we take a different approach that nonetheless bears some
similarity to the convex analytic method, in that we obtain a linearly
constrained nonlinear programming problem that is equivalent to a
generalization of Problem~(\ref{eq:intro_risk}).  We formulate the
problem as a \emph{``dynamic'' analytic problem}, in the sense that
the static adjoint equations of the convex analytic method are
replaced by a time-dependent equation, the Kolmogorov forward
equation. The forward equation yields the joint, time-dependent
distribution of the state of the controlled process and the associated
cumulative costs. This distribution is then in turn used to evaluate the
risk-aware objective that can now feature generic risk functions. The
dynamic formulation is natural to the risk-aware problem:
Risk-awareness generally requires in some way tracking running costs,
or future risks, given the information available to the controller at
any given time, see {e.g.} \cite{Isohatala2020_subm}, where we showed
that; Peng's nonlinear expectations \cite{Peng1997} also introduce an
additional process, modeling the controller's risks.


\subsection{Related literature}

\paragraph*{Risk measures}

There is a substantial body of work on risk measures in the static
setting, such as \cite{Ruszczynski06, kusuoka2001law,
  frittelli2014risk}.  This work focuses on axiomatic foundations for
modeling preferences, as well as for tractable risk-aware optimization
schemes. Dynamic risk functions are discussed in
\cite{Acciaio2011}. Nonlinear expectations form a subset of dynamic
risk functions, and are considered in \cite{Peng2004_NonlinE,
  RosazzaGianin2006}.

\paragraph*{The convex analytic approach}

The convex analytic method (or the linear programming method, in the
case that the problem is risk-neutral) is closest in spirit to the
approach we take in this paper. It has featured heavily in the study
of Markov decision processes (MDPs) and controlled stochastic processes.
In the discrete time setting, the linear programming approach for
MDPs is pioneered in \cite{Manne_Linear_1960} and further developed
in \cite{kallenberg1983linear}. An early survey of this technique
is found in \cite{ABF+93}. The main idea is that some MDPs can be
written as linear programming problems in terms of appropriate occupation
measures. A rigorous theory of the convex analytic approach for MDPs
with general Borel state and action spaces is developed in the works
\cite{borkar1988convex,HG98,Bor02,HL02}. Detailed monographs on Markov
decision processes are found in \cite{HL96,HL99,Put05}.

The convex analytic approach has also been well studied for continuous
time controlled Markov processes. Occupation measures for controlled
Markov processes in continuous time and state and action spaces were
first introduced in \cite{Stockbridge90a,Stockbridge90b}, where the
process dynamics were stated as a martingale problem and long term
average costs were considered. The theory was extended to discounted
and finite-horizon problems in the closely related papers
\cite{bhatt96} and \cite{Kurtz1998}, which also proved the optimality
of feedback controls (i.e. controls that depend only on the current
state). Convex analytic methods for controlled stochastic differential
equations are considered in \cite{borkar2005existence}. Singular
controls (see {e.g.} \cite{Shreve88} for an introduction) have
subsequently been analyzed within the convex analytic framework in
\cite{Taksar97} for diffusion processes with discounted
costs. Martingale problems with singular dynamics and controls, with
ergodic and discounted costs, were studied in \cite{Kurtz2001} and the
constrained case was studied in \cite{Kurtz2017}. The martingale
formulation of the problem and convex analytic methods were used in
the study of optimal stopping problems in \cite{Moon02}, and in
\cite{Helmes07} where also singular dynamics and controls were
included. Constrained continuous time MDPs are solved using convex
analytic techniques in \cite{guo2011discounted}, where the process
dynamics are described by a transition kernel rather than the
generator. More recently, a similar occupation measure approach for
controlled Markov jump processes is developed in
\cite{piunovskiy2011discounted, piunovskiy2015randomized}.  A survey
of optimal control methods for diffusion processes in particular can
be found in \cite{Borkar05}.

\subsection{Contributions}

The main contribution of this paper is developing a dynamic analytic
formulation of a generic risk-aware control problem. In particular,
(\emph{i}) we firstly state risk-aware control problems where the
controlled processes are described by martingale problems. We allow
for generic, Polish state and action space which makes our results
applicable for a broad family of types of stochastic processes;
continuous-time Markov decision processes and controlled L\'evy
processes are examples of these.  We require a number of rather
technical assumptions that are nonetheless often
satisfied. (\emph{ii}) Additionally, we introduce a number of
regularity conditions that ensure that the solutions of the martingale
problem are sufficiently well-behaved, {e.g.} in the sense that the
solutions never ``explode'' by diverging to some infinity point.
(\emph{iii}) We then derive our dynamic analytic formulation, and
prove its equivalence with the original martingale problem.
This is
based on a state space augmentation scheme, similar to the one in
\cite{bauerle14,haskell15}, that allows for the Kolmogorov forward
equation to also capture the distribution of costs. We additionally
provide conditions under which the optimal value is attained.

This paper is organized as follows. We begin in the next section by
introducing standard notation and describing the control model we
consider. This section defines our risk-aware problem, and states
the main assumptions. Section~\ref{sec:main} contains our main results,
where we show that Problem~(\ref{eq:intro_risk}) is equivalent to
a static optimization problem over measure-valued functions of time
satisfying a linear constraint (namely, the forward equation). In
Section~\ref{sec:numer} we present a simple application of the results.
Section~\ref{sec:concl} gives a short summary of the results. Some
of the proofs and frequently used auxilliary results are given in
the Appendix.

\section{Model}

\paragraph*{Basic definitions}

Let $\TT\deq\RR_{\geq0}$ or $[0,T]$ for some $T\in\RR_{>0}$ be
the set of time indices, and let $\RR_{\infty}\deq\RR\cup\{\infty\}$.
We shall cover both finite and infinite time-horizon problems; which
one we consider is determined whether $\TT$ is compact or not.

For any topological space $\UU$, we denote the Borel $\sigma$-algebra
on $\UU$ by $\bB(\UU)$. Finite Borel (probability) measures on $\UU$
are denoted $\Mm(\UU)$ ($\Pp(\UU)$). The space of probability measures
defaults to the topology of weak convergence, and for separable metric
space $\UU$, this topology is metrizable using the Prokhorov metric,
denoted $d_{P}$ \cite[Section 3.1]{EK1986}. Weak convergence of
$(\mu_{n})_{n\in\NN}\in\Pp(\UU)^{\NN}$ to a $\mu\in\Pp(\UU)$ is
denoted $\mu_{n}\Rightarrow\mu$. Given topological spaces $\UU_{1}$
and $\UU_{2}$, we say that a Borel measurable mapping $\pi:\UU_{2}\to\Pp(\UU_{1})$
is a transition function from $\UU_{2}$ to $\UU_{1}$, and denote
the set of transition functions from $\UU_{2}$ to $\UU_{1}$ by $\Pp(\UU_{1}\mid\UU_{2})$.

For a given probability space $(\Omega,\Sigma,\PP)$, we denote the
set of all $(\UU,\bB(\UU))$-valued random variables by $\Ll(\Omega,\Sigma,\PP;\UU)$
or $\Ll(\Omega;\UU)$ for short. The expectation with respect to $\PP$
is denoted by $\EE$. The law of a random variable $X\in\Ll(\Omega,\Sigma,\PP;\UU)$
is denoted $\lL(X)\deq\PP\circ X^{-1}$. For a Banach space $(\UU,|\cdot|)$,
by $\Ll^{p}(\Omega,\Sigma,\PP;\UU)$ or simply $\Ll^{p}(\Omega;\UU)$,
$p\in[1,\infty)$, we mean the set of $X\in\Ll(\Omega;\UU)$ such
that $\EE[|X|^{p}]<\infty$. The norms on the spaces $\Ll^{p}(\Omega;\RR)$,
$p\in[1,\infty]$, are denoted $\Vert\cdot\Vert_{p}$. For every $p\in[1,\infty)$
and Polish $(\UU,d)$, we use $\Pp^{p}(\UU)$ to denote the probability
measures such that for all $\mu\in\Pp^{p}(\UU)$, for some $u_{0}\in\UU$,
$\int d(u,u_{0})^{p}\mu(\D u)<\infty$. We assign $\Pp^{p}(\UU)$
the $p$-Wasserstein metric \cite[Definition 6.1]{Villani2009}, denoted
$W^{p}$.

For a pair of measurable spaces $\UU$ and $\VV$, measurable functions
from $\UU$ to $\VV$ are denoted $M(\UU,\VV)$, $B(\UU,\VV)$ if
they are bounded and $\VV$ is metric. Continuous functions shall
be the set $C(\UU,\VV)$ which is by default assigned the compact-open
topology. If $\VV=\RR$, the $\VV$ argument is omitted. Bounded and
continuous, and compactly supported continuous $\RR$-valued functions
are denoted $C_{b}(\UU)$ and $C_{c}(\UU)$, respectively, and these
are assigned the supremum norm, denoted $\Vert\cdot\Vert$. If $(\UU,d)$
is a metric space, bounded Lipschitz functions are denoted $C_{bl}(\UU)$,
and are defined so that $C_{bl}(\UU)\deq\{f\in C_{b}(\UU)\mid\Vert f\Vert_{bl}<\infty\}$,
where $\Vert\cdot\Vert_{bl}\deq\Vert\cdot\Vert+\Vert\cdot\Vert_{l}$
and $\Vert f\Vert_{l}\deq\sup_{u^{\prime}\neq u}|f(u^{\prime})-f(u)|/d(u^{\prime},u)$
for all $f\in C_{b}(\UU)$.

C\`adl\`ag, or left-continuous with limits from the right, functions from
$\TT$ to a Polish $\UU$ are denoted $D(\TT,\UU)$. For $\UU=\RR^{n}$,
$n\in\NN$ we use $C^{(k_{1},\ldots,k_{n})}(\UU)$ to denote functions
that can be differentiated $k_{i}$ times with respect to the $i$th
argument, $i\in\{1,\ldots,n\}$, with all the derivatives being in
$C(\UU)$, and similarly for the function spaces $C_{b}$ and $C_{c}$.

Let $\UU_{1}$ and $\UU_{2}$ be Polish spaces. For all $\mu\in\Pp(\UU_{1}\times\UU_{2})$
we denote the $\UU_{1}$, $\UU_{2}$ marginals of $\mu$ by $\mu^{\UU_{1}}$
and $\mu^{\UU_{2}}$, respectively. The regular conditional probabilities
on $\UU_{1}$ given $u_{2}\in\UU_{2}$ are denoted $\mu^{\UU_{1}\mid\UU_{2}}\in\Pp(\UU_{1}\mid\UU_{2})$
so that for all $f \in M(\UU_{1}\times\UU_{2})$,
\begin{gather*}
\int_{\UU_{1}\times\UU_{2}}f(u_{1},u_{2})\mu(\D u_{1},\D u_{2})=\int_{\UU_{2}}\left[\int_{\UU_{2}}f(u_{1},u_{2})\mu^{\UU_{1}\mid\UU_{2}}(\D u_{1}\mid u_{2})\right]\mu^{\UU_{2}}(\D u_{2}).
\end{gather*}
We will frequently need to separate measures into their marginal and
conditional parts, and hence we abbreviate equalities of the above
form to $\mu(\D u_{1}\times\D u_{2})=\mu^{\UU_{1}\mid\UU_{2}}(\D u_{1}\mid u_{2})\mu^{\UU_{2}}(\D u_{2})$.

Evaluation of a function $f$ defined on $\TT$ at a point $t\in\TT$
is denoted $f_{t}$.

We introduce a weak topology for functions $\mu\in M(\TT,\Pp(\UU))$,
where $\UU$ is Polish. We say that $(\mu^{(n)})_{n\in\NN}\in M(\TT,\Pp(\UU))^{\NN}$
\emph{converges} \emph{weakly} to a $\mu\in M(\TT,\Pp(\UU)$ and denote
$\mu^{(n)}\warrow\mu$, if and only if for all $h\in C_{b}(\UU\times\TT)$
such that the support of $h$ is contained in a set $\UU\times[0,t_{h}]$,
$t_{h}\in\TT$, we have that $\int_{\TT}\int_{\UU}h(u,t)\mu_{t}^{(n)}(\D u)\,\D t\to\int_{\TT}\int_{\UU}h(u,t)\mu_{t}(\D u)\,\D t$;
this is used in {e.g.} \cite{Kurtz2001}. For $C(\TT,\Pp(\UU))$, we
assume the (metrizable) topology of uniform convergence on compacts,
and denote $\mu^{(n)}\ccarrow\mu$ when a sequence $(\mu^{(n)})_{n\in\NN}\in C(\TT,\Pp(\UU))^{\NN}$
converges to a $\mu\in C(\TT,\Pp(\UU))$. Additionally, for any Polish
$\UU$ and $\VV$ and $\mu\in M(\TT,\Pp(\UU\times\VV))$, we denote
$\mu^{\UU}\deq(\mu_{t}^{\UU})_{t\in\TT}\in M(\TT,\Pp(\UU))$.

The Dirac measure centered at $u\in\UU$, $\UU$ a measurable space,
is denoted by $\delta_{u}$.

\subsection{Martingale formulation of the control problem}

In the following, $\XX$ and $\AA$ shall represent the state and
action spaces, both assumed Polish. We will also need to consider
processes on other (Polish) state spaces, and so, when appropriate
we state our definitions for a generic state space $\UU$.

The dynamics of the control problem are determined by the generator
of the process and an initial distribution. The following definition
formalizes these terms and introduces the notion of a solution that
we shall be using to describe the dynamics of our controlled processes.
\begin{defn}
\label{def:relaxed}(\emph{Relaxed controlled martingale problem})
Let $\UU$ and $\AA$ be Polish spaces, and let $A:\dD(A)\supset C_{b}(\UU)\to\rR(A)\subset C(\UU\times\AA\times\TT)$
and $\nu_{0}\in\Pp(\UU)$ be given.

We call the pair $(A,\nu_{0})$ a \emph{relaxed controlled martingale
problem}, where $A$ is the generator of the processes considered,
and $\nu_{0}$ is the initial distribution.

(\emph{Solution to a relaxed controlled martingale problem}) Let $(A,\nu_{0})$
be a relaxed controlled martingale problem. A \emph{solution to the
relaxed controlled martingale problem} $(A,\nu_{0})$ consists of
a filtered probability space $(\Omega,\Sigma,\Ff=(\Ff_{t})_{t\in\TT},\PP)$
and a $\UU\times\Pp(\AA)$-valued stochastic process $(u,\pi)=(u_{t},\pi_{t})_{t\in\TT}$
defined on $(\Omega,\Sigma,\Ff,\PP)$ such that: (\emph{i}) The process
$(u,\pi)$ is progressively measurable with respect to the filtration
$\Ff$; (\emph{ii}) the distribution of $u_{0}$ equals $\nu_{0}$;
and (\emph{iii}) for all $f\in\dD(A)$, the process $(m_{t}^{f})_{t\in\TT}$,
\begin{gather}
m_{t}^{f}\deq f(u_{t})-f(u_{0})-\int_{0}^{t}\int_{\AA}Af(u_{s},a,s)\pi_{s}(\D a)\,\D s\quad\forall f\in\dD(A),\,t\in\TT,\label{eq:martingale}
\end{gather}
is an $\Ff$-martingale for all $f\in\dD(A)$. We denote the set of
relaxed controlled solutions by $\RrR(A,\nu_{0})$, and for brevity,
we shall identify a solution by its control component, i.e. write
$\pi\in\RrR(A,\nu_{0})$ to mean $(\Omega,\Sigma,\Ff,\PP,u,\pi)$.

(\emph{C\`adl\`ag solution to a relaxed controlled martingale problem})
A solution $\pi\in\RrR(A,\nu_{0})$ is a \emph{c\`adl\`ag solution to
the relaxed controlled martingale problem} if additionally $u \in D(\TT,\UU)$,
$\PP$-almost surely. The subset of c\`adl\`ag solutions shall be denoted
$\DdD(A,K,\nu_{0})$.
\end{defn}

We allow constraints on the relaxed controlled solutions that only
depend on the finite dimensional distributions of controls and states.
\begin{defn}
\label{def:admissible}Let $(A,\nu_{0})$ be a relaxed controlled
problem and $K\subset M(\TT,\Pp(\UU\times\AA))$. A relaxed controlled
solution $\pi\in\RrR(A,\nu_{0})$ is \emph{admissible} (given $K$)
if $\mu\in K$, where $\mu$ is defined
\begin{gather*}
\int_{\UU\times\AA}h(u,a)\mu_{t}(\D u\times\D a)=\EE\biggl[\int_{\AA}h(u_{t}^{\pi},a)\pi_{t}(\D a)\biggr]\quad\forall h\in C_{b}(\UU\times\AA),\,t\in\TT.
\end{gather*}
Constrained problems and the associated solutions are denoted $(A,K,\nu_{0})$
and $\RrR(A,K,\nu_{0})$, $\DdD(A,K,\nu_{0})$, respectively.
\end{defn}

We emphasize that each relaxed controlled solution $\pi\in\RrR(A,\nu_{0})$
comes in general with its own filtered probability space. When appropriate,
we label the objects forming the solution as $(\Omega^{\pi},\Sigma^{\pi},\Ff^{\pi},\PP^{\pi},(u^{\pi},\pi))$
to make this point explicit. In the following, we shall consider almost
exclusively c\`adl\`ag solutions.

\paragraph*{Baseline assumptions on the relaxed controlled problem}

First, we introduce a few technical definitions that are necessary
to state our main assumptions. We recall the notion of pre-generators,
used to characterize the operators that are sufficiently regular to
correspond to generators of Markov processes \cite{Kurtz2001}:
\begin{defn}
Let $\UU$ be a Polish space. An operator $A:M(\UU)\to M(\UU)$ is
a \emph{pre-generator} if it is: (\emph{i}) dissipative, i.e. for
all $\lambda>0$ and all $f\in\dD(A)$, $\Vert(\lambda-A)f\Vert\geq\lambda\Vert f\Vert$,
and (\emph{ii}) there are sequences of measure valued functions $(\mu_{n})_{n\in\NN}$
with $\mu_{n}:\UU\to\Pp(\UU)$ and $(\lambda_{n})_{n\in\NN}$ with
$\lambda_{n}:\UU\to\RR_{\geq0}$, for all $n\in\NN$, such that $h(u)=\lim_{n\to\infty}\lambda_{n}(u)\int_{E}(f(u)-f(u'))\mu_{n}(u)(\D u')$
for all $u\in\UU$ and for every $f\in\dD(A)$, $h\in\rR(A)$ such
that $Af=h$.
\end{defn}

We also utilize the notion of bounded point-wise limit and strong
separability of points, see {e.g.} \cite[Chapter 3.4]{EK1986}.
\begin{defn}
Let $\UU$ be a metric space. (\emph{i}) A sequence of functions $(f_{k})_{k\in\NN}\subset B(\UU)$
converges \emph{boundedly and point-wise} to a function $f\in B(\UU)$
if $\sup_{k\in\NN}\Vert f_{k}\Vert<\infty$ and $\lim_{k\to\infty}f_{k}(u)=f(u)$
for all $u\in\UU$. We denote this $\bplim_{k\to\infty}f_{k}=f$.
(\emph{ii}) A set $M\subset B(\UU)$ is said to be \emph{bp-closed},
if for all $(f_{k})_{k\in\NN}\subset M$, $\bplim_{k\to\infty}f_{k}=f\in B(\UU)$
implies $f\in M$. (\emph{iii}) The \emph{bp-closure} of a set $M\subset B(\UU)$
is the smallest bp-closed set that contains $M$. (\emph{iv}) A set
of functions $\Aa\subset C_{b}(\UU)$ is said to \emph{strongly separate
points }if for every $u\in\UU$ and a neighborhood $U$ of $u$, there
is a finite $\Aa(u,U)\subset\Aa$ such that $\inf_{u^{\prime}\notin U}\max_{f\in\Aa(u,U)}|f(u^{\prime})-f(u)|>0$.
\end{defn}

The following assumption, adapted from \cite{Kurtz2001,Kurtz2017},
is used to guarantee existence of relaxed solutions to controlled
martingale problems, as stated below in Theorem~\ref{thm:ks}.
\begin{assumption}
\label{cond:hasmarkov}Let $\UU$ and $\AA$ be Polish spaces, and
let $A:C_{b}(\UU)\supset\dD(A)\to\rR(A)\subset C(\UU\times\AA\times\TT)$.
The tuple $(\UU,\AA,A)$ satisfies the following conditions:

(\emph{i}) The constant function $1\in C_{b}(\UU)$ is in $\dD(A)$
and $A1=0$.

(\emph{ii}) The operator $A_{a}$ defined as $A_{a}f(u,t)\deq Af(u,a,t)$
for all $f\in\dD(A)$ and $a\in\AA$ is a pre-generator.

(\emph{iii}) The domain of $A$, $\dD(A)$, is an algebra that strongly
separates points. 

(\emph{iv}) There is a function $\psi\in C(\UU\times\AA\times\TT)$,
$\psi\geq1$, such that for each $f\in\dD(A)$ there is a constant
$a_{f}$ satisfying $|Af(u,a,t)|\leq a_{f}\psi(u,a,t)$ for all $(u,a,t)\in\UU\times\AA\times\TT$.

(\emph{v}) The set $A_{0}\deq\{(f,\psi^{-1}Af)\mid f\in\dD(A)\}$
is such that there exists $\{f_{k}\}_{k\in\NN}\subset\dD(A)$ for
which $A_{0}$ is contained in the bp-closure of the linear span of
$\{(f_{k},A_{0}f_{k})\}_{k\in\NN}$.
\end{assumption}

Parts (\emph{i})--(\emph{iii}) in Assumption~\ref{cond:hasmarkov}
amount to basic requirements for the martingale problem and its
associated forward equation to have solutions (compare to the
standard, though stronger assumptions of Theorem~4.5.4 and
Theorem~4.9.19 in \cite{EK1986} in the uncontrolled case, with a
locally compact state space $\UU$).  The requirement that $A$ is a
pre-generator is a relaxation of the assumption that $A$ satisfies the
positive maximum principle. Part (\emph{iv}) of
Assumption~\ref{cond:hasmarkov} allows for construction of an
operator, specifically $\psi^{-1}A$, that takes values on bounded
continuous functions, and which is used in weak convergence arguments.
Part (\emph{v}) is used in \cite{Kurtz2001} to construct a compact
Polish space $\hat{\UU}$ along with a continuous mapping
$\Gamma:\UU\to\hat{\UU}$ with a measurable inverse that allows
extending of results assuming a compact state space to the case where
$\UU$ is not compact or locally compact. This condition was earlier
applied in \cite{Bhatt1993} for the same purpose in the context of
uncontrolled martingale problems and in \cite{bhatt96} for controlled
problems. Additional discussion and examples can be found in
\cite{Kurtz2001,Kurtz2017}. Returning to part (\emph{iii}), we note
that typically it is assumed that $\dD(A)$ only separates
points. Here, we assume strong separation of points, and this is to
ensure that the above mapping $\Gamma$ is in fact a homeomorphism
(that is, its inverse is also continuous) \cite[Lemma 1]{Blount2010}.
A convenient characterization of sets that strongly separate points is
given in \cite[Lemma 4]{Blount2010}. We also recall that sets that
strongly separate points are convergence determining \cite[Theorem
  3.4.5(b)]{EK1986}.

In order to establish the equivalence of control problems stated in
terms of relaxed controlled solutions and those formulated using analytic
methods, we will require additional constraints on the generator $A$.
\begin{defn}
\label{def:regular}Suppose $(\UU,\AA,A)$ satisfies Assumption~\ref{cond:hasmarkov}.
We say the martingale problem $(A,K,\nu_{0})$ is \emph{regular},
if there exists constants $L_{1},L_{\UU},L_{\AA}>0$, $\beta_{1}>1$,
and $\Lambda_{1},\text{\ensuremath{\Lambda}}_{\AA}\geq0$, non-negative
functions $\phi=(\phi_{n})_{n\in\NN}\subset\dD(A)$, and $\psi_{\UU}\in C(\UU)$
and $\psi_{\AA}\in C(\AA)$ such that (\emph{i}) $|A\phi_{n}(u,a,t)|\leq\Lambda_{1}(1+\psi_{\UU}(u)+\psi_{\AA}(a))$
and $\psi(u,a,t)^{\beta_{1}}\leq L_{1}(1+\psi_{\UU}(u)+\psi_{\AA}(a))$
for all $(u,a,t)\in\UU\times\AA\times\TT$ and $n\in\NN$; (\emph{ii})
the sequence $(\phi_{n})_{n\in\NN}$ is increasing and converges pointwise
to $\psi_{\UU}$; (\emph{iii}) $\psi_{\UU}$ and $\psi_{\AA}$ are
inf-compact; (\emph{iv}) the initial distribution satisfies
\begin{gather}
\int\psi_{\UU}(x)\nu_{0}(\D x)\leq L_{\UU};\label{eq:regular-iv-bound}
\end{gather}
(\emph{v}) for all $\mu\in K$, 
\begin{gather}
\int_{\AA}\psi_{\AA}(a)\mu_{t}^{\AA}(\D a)\leq L_{\AA}\E^{\text{\ensuremath{\Lambda}}_{\AA}t}\quad\forall t\in\TT;\label{eq:regular-control-bound}
\end{gather}
and (\emph{vi}) $K$ is closed in the weak topology.
\end{defn}

As a notational aside, we use the symbol $\Lambda$ for quantities
representing exponential growth rates, $L$ for bounds and constants
of proportionality, and $\beta$ for powers that control relative
magnitudes and scaling rates between different quantities; naturally,
the $\Lambda$'s are the most important, while the $L$'s tend to
be the least significant.

The condition that a relaxed controlled martingale problem is regular
can be viewed as a generalization of growth bounds on {e.g.} the solutions
of stochastic differential equations, for which it is common to assume
that the drift and diffusion coefficients have at most linear growth.
The following example illustrates this.
\begin{example}
\label{exa:sde}Consider a stochastic differential equations driven
by orthogonal martingale measures, see {e.g.} \cite{ElKaroui1990},
on $\XX=\RR^{d_{x}}$, $\AA=\RR^{d_{a}}$, $d_{x},d_{a}\in\NN$, characterized
by drift and diffusion functions $b\in C(\XX\times\AA\times\TT,\XX)$
and $\sigma\in C(\XX\times\AA\times\TT,\XX\times\XX)$. Suppose $b$
and $\sigma$ have bounded growth in the sense that $\left|b(x,a,t)\right|,\left|\sigma(z,a,t)\right|\leq L(1+|x|+|a|^{q})$,
$q\in\RR_{\geq0}$, for some $L>0$ and all $(x,a,t)\in\XX\times\AA\times\TT$
($|\cdot|$ stands for the Frobenius norm for matrices). The corresponding
generator reads 
\begin{align*}
Gf(x,a,t) & \deq b(x,a,t)^{\top}\nabla f(x)+\frac{1}{2}\trace\left\{ \sigma\sigma^{\top}(x,a,t)\nabla^{\top}\nabla f(x)\right\} \\
 & \qquad\quad\forall f\in C_{c}^{(2)}(\XX),\,(x,a,t)\in\XX\times\AA\times\TT,
\end{align*}
and where $\nabla$ and $\nabla^{\top}\nabla$ stand for the gradient
and Hessian operators, respectively. The domain of $G$ can be taken
to be $\dD(G)=\{f+f_{0}\mid f\in C_{c}^{(2)}(\XX),f_{0}\in\RR\}$.
The regularity conditions are satisfied {e.g.} with the choices $\psi(x,a,t)=1+|a|^{2q}$,
$\psi_{\XX}(x)=|x|^{2}$, $\psi_{\AA}(a)=|a|^{2q\beta_{1}}$ for all
$(x,a,t)\in\XX\times\AA\times\TT$, and where $\beta_{1}>1$ can be
arbitrarily small. Additionally, we can take
\begin{gather*}
a_{f}=\bar{a}\biggl\{\Bigl\Vert\bigl(1+\left|\cdot\right|\bigr)\nabla f(\cdot)\Bigr\Vert+\Bigl\Vert\bigl(1+\left|\cdot\right|^{2}\bigr)\nabla^{\top}\nabla f(\cdot)\Bigr\Vert\biggr\}\quad\forall f\in\dD(G),
\end{gather*}
where $\bar{a}$ is a constant independent of $f$. The initial distribution
should now have finite variance, by Eq.~(\ref{eq:regular-iv-bound}),
and the coefficients on the right-hand side of Eq.~(\ref{eq:regular-control-bound})
can be selected freely. We note that albeit $\Lambda_{\AA}$ and $\beta_{1}$
may be chosen arbitrarily large and small, respectively, there will
be a trade-off, formalized later in Assumption~\ref{assu:existence}.
\end{example}

The regularity assumptions guarantee that, almost surely, a c\`adl\`ag
solution never explodes in the sense that, almost surely, $\psi_{\UU}(u_{t})$
is finite for all $t\in\TT$.
\begin{prop}
\label{prop:regular-means-regular}Suppose $(A,K,\nu_{0})$ is regular,
with $\psi_{\UU}$ and $\psi_{\AA}$ as in Definition~\ref{def:regular}.
Then for all $\pi\in\DdD(A,K,\nu_{0})$, $\int_{\AA}(1+\psi_{\UU}(u_{t})+\psi_{\AA}(a))\pi_{t}(\D a)<\infty$
for all $t\in\TT$, $\PP$-almost surely.
\end{prop}

The regularity requirement is important, as it constrains the problems
we consider to those with well-behaved trajectories. While weaker
assumptions were used in the treatment of risk-neutral problems in
{e.g.} \cite{bhatt96,Kurtz1998}, our approach describes the costs associated
with each relaxed controlled solution via their distributions as given
by the forward equation, and for validity of this approach, a higher
degree of regularity is necessary.
\begin{rem}
We note the difference between the functions $\psi$, as given in
Assumption~\ref{cond:hasmarkov} and $(\psi_{\UU},\psi_{\AA})$,
given in Definition~\ref{def:regular}. The former describes how
large the functions in the range of the generator may be, while the
latter characterize how large values the solutions themselves may
take, cf. the bound given by Proposition~\ref{prop:regular-means-regular}.
\end{rem}

\subsection{Risk-aware objectives\label{sec:risk}}

Given a relaxed controlled problem $(G,K,\nu)$ on Polish state and
action spaces $\XX$ and $\AA$, we then suppose we are also provided
a cost rate function $c\in C(\XX\times\AA\times\TT)$, and in the
case of finite-horizon problems, a terminal cost function $v\in C(\XX)$.
In addition, we suppose we are given a risk function $\rho :
\Ll(\Omega;\RR) \to \RR_\infty$ that is defined on some reference
probability space $(\Omega, \Sigma, \PP)$. Since relaxed controlled
solutions in general come with their own probability spaces, we make
the restriction to law-invariant risk functions, so that the problem
is well-defined.
\begin{defn}
\label{def:law-invariance}Let $(\Omega,\Sigma,\PP)$ be a probability
space. A mapping $\rho\text{ : }\Ll(\Omega;\RR)\rightarrow\RR_{\infty}$
is \emph{law invariant }if there exists a function $\tilde{\rho}:\Pp(\RR)\to\RR_{\infty}$
such that $\rho(X)=\tilde{\rho}(\lL(X))$ for all $X\in\Ll(\Omega;\RR)$.
\end{defn}
The requirement that the risk functions are law invariant is very
mild, and is in practice essentially always satisfied.

Any law invariant risk function $\rho$ defined on random variables
of some fixed probability space $(\Omega,\Sigma,\PP)$ can be used
to evaluate the risk of random variables on any other $(\Omega^{\prime},\Sigma^{\prime},\PP^{\prime})$
by setting $\rho(X^{\prime})\deq\tilde{\rho}(\lL(X^{\prime}))$ for
all $X^{\prime}\in\Ll(\Omega^{\prime},\Sigma^{\prime},\PP^{\prime};\RR)$.
For law invariant risk functions we can then define the risk-aware
problem, Problem $\pPM$, as 
\begin{gather}
\begin{gathered}\inf_{\pi\in\DdD(G,K,\nu)}\limsup_{t\to\infty}\rho\left(\int_{0}^{t}\E^{-\alpha s}\int_{\AA}c(x_{s}^{\pi},a,s)\pi_{s}(\D a)\,\D s\right)\qquad(\TT=\RR_{\geq0}),\\
\inf_{\pi\in\DdD(G,K,\nu)}\rho\left(\int_{0}^{T}\int_{\AA}c(x_{s}^{\pi},a,s)\pi_{s}(\D a)\,\D s+v(x_{T}^{\pi})\right)\qquad(\TT=[0,T]).
\end{gathered}
\label{eq:ra-objective}
\end{gather}

By Definition~\ref{def:law-invariance}, a law invariant risk function
$\rho:\Ll(\Omega;\RR)\to\RR_{\infty}$ can be equivalently expressed
using a functional $\tilde{\rho}:\Pp(\RR)\to\RR_{\infty}$. Since our
dynamic analytic formulation constructs directly the distribution of
the input random variable representing total costs, it will sometimes
be more natural to consider the risk function as a functional on
distributions rather than random variables. We note that the
literature on risk functions typically favors the picture of a risk
function as functional on random variables. Indeed, properties of risk
functions such as coherence and convexity, important from both
practical applications and theoretical analysis points of view
\cite{Artzner99,Follmer02,Frittelli02}, are conventionally defined for
$\rho$ viewed as mappings from $\Ll(\Omega;\RR)$ to
$\RR_{\infty}$. Analogous properties can be defined for risk functions
on probability measures, or equivalently, for $\tilde{\rho}$
\cite{frittelli2014risk}, but in general, {e.g.} the convexity
properties of $\rho$ and $\tilde{\rho}$ can be very different. In
fact, convex risk functions generally have representations on measures
that are concave \cite{Acciaio2013}. Here, we shall not consider
questions such as the uniqueness of solutions, and we do not require
convexity of the risk functions.

Our baseline assumptions are then as follows.
\begin{assumption}
\label{assu:main}Let $\XX$ and $\AA$ be given Polish state and
action spaces, with $d$ denoting the metric on $\XX$. (\emph{i})
The generator $G:C_{b}(\XX)\supset\dD(G)\to\rR(G)\subset C(\XX\times\AA\times\TT)$,
admissible solutions $K$, and the initial distribution $\nu\in\Pp(\XX)$
are such that the relaxed controlled martingale problem $(G,K,\nu)$
is regular; (\emph{ii}) the cost rate function $c\in C(\XX\times\AA\times\TT)$
is non-negative, and there are $L_{c}>0$ and $\beta_{c}\geq\beta_{1}>1$
such that $c^{\beta_{c}}\leq L_{c}(1+\psi_{\XX}+\psi_{\AA})$; (\emph{iii})
if a finite time-horizon problem is considered, then we have a terminal
cost function $v\in C(\XX)$ that is non-negative, else we are given
a discount rate $\alpha>0$; (\emph{iv}) the risk function is law
invariant.
\end{assumption}

We will later require continuity of the risk functions, and in particular,
continuity of its representation on measures. The following shows
that if a risk function is continuous on random variables, then it
is continuous on measures, and similarly for lower semicontinuity. 
\begin{prop}
\label{prop:rho-rho-tilde-lsc}Let $(\Omega,\Sigma,\PP)$ be a probability
space, $\rho:\Ll^{p}(\Omega;\RR)\to\RR$, $p\in[1,\infty)$, and let
$\tilde{\rho}:\Pp^{p}(\RR)\to\RR$ be such that $\rho(X)=\tilde{\rho}(\lL(X))$
for all $X\in\Ll^{p}(\Omega;\RR)$. If $\rho$ is continuous (respectively
lower semicontinuous) in the strong, $\Vert\cdot\Vert_{p}$-norm topology,
then $\tilde{\rho}$ is continuous (respectively lower semicontinuous)
in the topology induced by the $p$-Wasserstein metric.
\end{prop}

Continuity holds for many common risk functions. Indeed, convex risk
functions $\rho:\Ll^{p}(\Omega;\RR)\to\RR$ are norm-continuous \cite{Ruszczynski06},
and hence their representations in terms of functionals over measures
are also continuous.
\begin{example}
Returning to the problem of Example~\ref{exa:sde}, we can now consider
cost rate functions that satisfy Assumption~\ref{assu:main}. In
particular, the cost rate function $c(x,a,t)\deq1+|x|^{q_{1}}+|a|^{q_{2}}$,
or anything bound by this, for all $(x,a,t)\in\XX\times\AA\times\TT$
is admissible, if $q_{1}\leq2/\beta_{c}$, $q_{2}\leq2q\beta_{1}/\beta_{c}$
for some $\beta_{c}\geq\beta_{1}>1$. As examples of law invariant
risk functions, we mention here the entropic risk function $\rho^{\text{Ent}}:\Ll(\Omega;\RR)\to\RR_{\infty}$,
and the mean semi-deviation risk function $\rho^{\text{MD+}}:\Ll(\Omega;\RR)\to\RR_{\infty}$.
For an arbitrary reference probability space $(\Omega,\Sigma,\PP)$,
these are defined for any $X\in\Ll(\Omega;\RR)$ as
\begin{gather*}
\rho^{\text{Ent}}(X)\deq\frac{1}{\theta}\ln\EE\left[\E^{\theta X}\right],\\
\rho^{\text{MD+}}(X)\deq\EE\left[X\right]+\beta\EE\left[(X-\EE\left[X\right])_{+}\right],
\end{gather*}
and where $\theta\in(0,\infty)$ and $\beta\in[0,1]$ are parameters.
These have the following representations as functions on probability
measures: For all $\mu\in\Pp(\RR)$,
\begin{gather*}
\tilde{\rho}^{\text{Ent}}(\mu)\deq\frac{1}{\theta}\ln\left(\int\E^{\theta x}\mu(\D x)\right),\\
\tilde{\rho}^{\text{MD+}}(\mu)\deq\int x\mu(\D x)+\beta\int\left(x-\int x^{\prime}\mu(\D x)\right)_{+}\mu(\D x).
\end{gather*}
Other examples would include {e.g.} mean-variance risk functions, and
the conditional value-at-risk.
\end{example}

\section{Dynamic analytic formulation\label{sec:main}}


We can now construct our dynamic analytic formulation of the problem.
The first step is to find evolution equations for the joint
distribution of the controlled processes state and accumulated costs.

\paragraph*{Forward equation and time-dependent distributions}


The main tool for finding the time-dependent distribution
of a stochastic process is the Kolmogorov forward equation, which
we shall discuss next.
\begin{defn}
\label{def:kfe}We say that $\mu\in M(\TT,\Pp(\UU\times\AA))$ \emph{satisfies
the forward equation for initial condition $\nu_{0}\in\Pp(\UU)$ and
generator $A:C_{b}(\UU)\supset\dD(A)\to\rR(A)\subset C(\UU\times\AA\times\TT)$}
if (we recall our notation where superscripts on measures indicate
taking marginals)
\begin{align}
\int_{\UU}f(u)\mu_{t}^{\UU}(\D u)-\int_{\UU}f(u)\nu_{0}(\D u) & =\int_{0}^{t}\int_{\UU\times\AA}Af(u,a,s)\mu_{s}(\D u\times\D a)\,\D s, \label{eq:KFE}
\end{align}
for all $f\in\dD(A)$ and $t\in\TT$.  We use $\FfF(A,\nu_{0})\subset
M(\TT,\Pp(\UU\times\AA))$ to denote the set of solutions of
Eq.~(\ref{eq:KFE}) and constrained solutions of Eq.~(\ref{eq:KFE}) are
defined analogously to Definition~\ref{def:admissible}:
$\FfF(A,K,\nu_{0})\deq\FfF(A,\nu_{0})\cap K$, where $K\subset
M(\TT,\Pp(\UU\times\AA))$ is again the set of admissible solutions.
\end{defn}

\paragraph*{Cost distribution}

To evaluate a law invariant risk function appearing in the objective,
we need means for finding the distribution of the costs appearing
in Eq.~(\ref{eq:ra-objective}). The forward equation provides the
distribution of the state variables, and the same equation can be
co-opted to additionally yield the cost distribution. This is done
by introducing an extended forward equation corresponding to a given
martingale problem $(G,K,\nu)$ that gives the joint distribution
of the state and running costs, that is, cost accumulated up to a
given time $t\in\TT$.

We define $\YY\deq\RR_{\geq0}$ to stand for the state space of the
running costs, and consider the original state and the running costs
in parallel on the space $\XX\times\YY$. The equation for the joint
distribution of states and costs shall be the forward equation corresponding
to a new generator $H$, describing the joint evolution of the states
and costs: Let $c\in C(\XX\times\AA\times\TT)$ be the continuous
cost rate function, and let $\alpha\in\RR_{>0}$ ($\TT=\RR_{\geq0}$)
or $\alpha=0$ ($\TT=[0,T${]}) be the discount rate. For the given
generator $G:C_{b}(\XX)\supset\dD(G)\to\rR(G)\subset C(\XX\times\AA\times\TT)$,
we define $H:C_{b}(\XX\times\YY)\supset\dD(H)\to\rR(H)\subset C(\XX\times\YY\times\AA\times\TT)$
via 
\begin{gather}
  \begin{aligned}
    Hfg(x,y,a,t) &\deq g(y)Gf(x,a,t)+\E^{-\alpha t}c(x,a,t)\frac{\partial g}{\partial y}(y) \\
    & \qquad \forall(x,y,a,t)\in\XX\times\YY\times\AA\times\TT,\,fg\in\dD(H), \\
    \dD(H) &\deq\bigl\{ f(g + g_0) \bigm|f\in\dD(G),\,g\in C_{c}^{(1)}(\YY),\,g_0\in\RR\bigr\}.
  \end{aligned}
  \label{eq:augmented-generator}
\end{gather}
Recalling that we took $\nu\in\Pp(\XX)$ as the initial distribution
for the $\XX$-space process, we define $\upsilon\deq\nu\times\delta_{0}\in\Pp(\XX\times\YY)$,
where $\delta_{0}$ is the natural, point mass initial distribution
of the $\YY$-space process, as the initial distribution for the augmented
process.

For each $\pi\in\RrR(G,K,\nu)$, we associate a real-valued running
costs process $y^{\pi}=(y_{t}^{\pi})_{t\in\TT}$, defined
\begin{gather}
  y_{t}^{\pi}\deq\int_{0}^{t}\E^{-\alpha s}\int_{\AA}c(x_{s}^{\pi},a,s)\pi_{s}(\D a)\,\D s\quad\forall t\in\TT,\label{eq:running-costs}
\end{gather}
where $\alpha=0$ if $\TT=[0,T]$. The following theorem states that
under our baseline assumptions, considering c\`adl\`ag relaxed controlled
solutions $\DdD(G,K,\nu)$ together with costs $y^{\pi}$ as defined
in Eq.~(\ref{eq:running-costs}), is equivalent to considering solutions
to the forward equation for joint, time-dependent state-cost distributions,
$\FfF(H,K,\upsilon)$. That is, a solution for either (\emph{i}) the
martingale problem with running costs or (\emph{ii}) the extended
forward equation problem, can be used to construct a solution for
the other problem type. For brevity, we are using formally the same
set of admissible solutions $K\subset M(\TT,\Pp(\XX\times\AA))$ for
both problems; in $\FfF(H,K,\upsilon)$ the constraints are assumed
to hold for the $\XX\times\AA$-marginals of the $M(\TT,\Pp(\XX\times\YY\times\AA))$
solutions.
\begin{thm}
\label{thm:d-f-equivalence}Suppose Assumption~\ref{assu:main} holds,
so that $(G,K,\nu)$ is a regular controlled martingale problem on
the state-action space $\XX\times\AA$.

(\emph{i}) If $\mu\in\FfF(H,K,\upsilon)$, then there exists a c\`adl\`ag
relaxed controlled solution $\pi\in\DdD(G,K,\nu)$ and a cost process
$y^{\pi}$ defined by Eq.~(\ref{eq:running-costs}), such that the
finite dimensional distributions of $(x^{\pi},y^{\pi})$ are given
by $\mu^{\XX\times\YY}$, and with the control process satisfying
$\pi_{t}=\mu_{t}^{\AA\mid\XX\times\YY}(\cdot\mid x_{t},y_{t})$ for
all $t\in\TT$.

(\emph{ii}) If $\pi\in\DdD(G,K,\nu)$ and $y^{\pi}$ is the associated
costs process of Eq.~(\ref{eq:running-costs}), then $\mu_{t}\in M(\TT,\Pp(\XX\times\YY\times\AA))$
defined
\begin{align}
  \int_{\XX\times\YY\times\AA}h(x,y,a)\mu_{t}(\D x\times\D y\times\D a)
  = \EE\left[\int_{\AA}h(x_{t}^{\pi},y_{t}^{\pi},a)\pi_{t}(\D a)\right],\label{eq:equiv-converse}
\end{align}
for all $h\in C_{b}(\XX\times\YY\times\AA)$ and $t\in\TT$ is a
solution $\mu\in\FfF(H,K,\upsilon)$.
\end{thm}

The proof is deferred to the second half of this section. We can now
move on to state our main results.

\paragraph*{Main results}

We define the dynamic analytic problem, Problem~$\pPC$ as 
\begin{gather*}
\inf_{\mu\in\FfF(H,K,\upsilon)}\limsup_{t\to\infty}\tilde{\rho}\left(\mu_{t}^{\YY}\right)\qquad(\TT=\RR_{\geq0}),\\
\inf_{\mu\in\FfF(H,K,\upsilon)}\tilde{\rho}\left(\mu_{T}^{\XX\times\YY}\circ\Theta^{-1}\right)\qquad(\TT=[0,T]),
\end{gather*}
where $\Theta(x,y)\deq y+v(x)$ for all $(x,y)\in\XX\times\YY$, $\tilde{\rho}:\Pp(\RR)\to\RR$
is such that for the given risk function $\rho:\Ll(\Omega;\RR)\to\RR$,
$\rho(X)=\tilde{\rho}(\lL(X))$ for all $X\in\Ll(\Omega;\RR)$.

Two theorems comprise our main results. The first states that under
our baseline assumptions, Problems~$\pPM$ and~$\pPC$ are equivalent,
and optimal controls are Markov in the state, running costs, and time.
\begin{thm}
\label{thm:equivalence}If Assumption~\ref{assu:main} holds, then
the optimal values of Problems~$\pPM$ and~$\pPC$ are equal. If
there is a $\pP_{\Pp}$-optimal $\mu\in\FfF(H,K,\upsilon)$, then
there exists a $\pP_{\Ll}$-optimal $\pi\in\DdD(G,K,\nu)$ such that
$\pi_{t} = \mu_{t}^{\AA\mid\XX\times\YY}(\cdot\mid x_t^\pi,y_t^\pi)$
for all $(t,x,y)\in\TT\times\XX\times\YY$. That is, the control $\pi$
is Markov, depending only on time, state, and running costs.
\end{thm}

While Theorem~\ref{thm:equivalence} guarantees that the forward
equation formulation, Problem~$\pPC$, yields the same optimal value
as solutions of Problem~$\pPM$, it does not establish the existence
of solutions. The following theorem and our second main result gives
sufficient conditions for there to be a $\mu\in\FfF(H,K,\upsilon)$
that attains the optimal value, provided the next assumptions hold.
\begin{assumption}
\label{assu:existence}Assumption~\ref{assu:main} holds, and additionally:
(\emph{i}) $\dD(G)\subset C_{bl}(\XX)$, and if $a_{f}$ is as in
Assumption~\ref{cond:hasmarkov}(\emph{iv}), then $f\to a_{f}$ defines
a seminorm on $\dD(G)$ and $a_{f}\geq\Vert f\Vert_{l}$, the Lipschitz
constant, for all $f\in\dD(G)$; (\emph{ii}) in the infinite time-horizon
case, the discount rate satisfies $\alpha>(\Lambda_{1}+\text{\ensuremath{\Lambda}}_{\AA})/\beta_{c}$;
(\emph{iii}) either (\emph{a}), the risk function $\rho:\Ll(\Omega;\RR)\to\RR_{\infty}$
is bounded from below, and continuous and coercive on $\Ll^{p}(\Omega;\RR)$
for some $p\in[1,\infty)$, that is, $\Vert X\Vert_{p}\to\infty$
implies $\rho(X)\to\infty$, or (\emph{b}), the cost rate function,
and the terminal cost function if $\TT=[0,T]$, are bounded, the risk
function $\rho$ is finite for compactly supported random variables,
and its representation $\tilde{\rho}$ on measures is continuous in
the topology of weak convergence.

If a finite time horizon problem is considered, then the continuity
of $\rho$ or $\tilde{\rho}$ may be replaced by lower semicontinuity.
\end{assumption}

The condition that $\dD(G)$ is a subset of bounded Lipschitz functions
and that $a_{f}$ is a seminorm bounded by $\Vert\cdot\Vert_{l}$
is used to construct a metric on probability measures that allows
us to prove uniform convergence and equicontinuity of families of
solutions to the forward equation. Note that {e.g.} the $a_{f}$ obtained
in Example~\ref{exa:sde} is indeed a seminorm bounded from below
by the Lipschitz constants. The lower bound on the discount rate $\alpha$
is needed to ensure that the cost distributions become stationary
as time tends to infinity. Intuitively, the exponent $\beta_{c}$
describes how fast the cost rate $c$ grows relative to the growth
of the solutions, represented by $\psi$, cf. Assumption~\ref{assu:main}.
The part $\Lambda_{1}+\text{\ensuremath{\Lambda}}_{\AA}$ in turn
gives the growth rate of $\psi_{\XX}+\psi_{\AA}$, as given by Definition~\ref{def:regular}(\emph{i}).
Hence, the inequality describes the balance between the growth of
costs and the rate of discounting; satisfying it guarantees that the
risks converge rather than oscillate as time tends to infinity. Continuity
or lower semicontinuity of the risk function is naturally necessary,
as we will be taking limits of minimizing sequences. Part (\emph{iii})
of the assumption is split into (\emph{a}) and (\emph{b}) alternatives
and the latter case is included to accommodate risk functions defined
on essentially bounded random variables, that is, the case where $p=\infty$.
\begin{thm}
\label{thm:existence}Suppose Assumption~\ref{assu:existence} holds,
and let $\rho^{\ast}\in\RR_{\infty}$ be the $\pP_{\Pp}$-optimal
value. If $\FfF(H,K,\upsilon)\neq\emptyset$, then there is a $\mu\in\FfF(H,K,\upsilon)$
that attains $\rho^{\ast}$.
\end{thm}

The result of Theorem~\ref{thm:existence} immediately implies that
the corresponding relaxed controlled martingale problem, $(G,K,\nu)$,
has an optimal solution for which the control process is Markov in
time, state, and running costs.

\paragraph*{Proofs of main results}

We begin with the proof of Theorem~\ref{thm:d-f-equivalence}. To
this end, we first give a pair of auxiliary results, first one stating
that c\`adl\`ag relaxed controlled solutions to $(G,K,\nu)$, together
with the associated costs processes, are in a sense equivalent to
c\`adl\`ag relaxed controlled solutions to the augmented problem, $(H,K,\upsilon)$.
Some proofs are deferred to the Appendix.
\begin{prop}
\label{prop:augmentation-equivalence}Suppose Assumption~\ref{assu:main}
holds and $\pi$ is a c\`adl\`ag relaxed controlled solution to the problem
$(H,K,\upsilon)$, $\pi\in\DdD(H,K,\upsilon)$. Then $(\Omega^{\pi},\Sigma^{\pi},\Ff^{\pi},\PP^{\pi},x^{\pi},\pi)$
is a c\`adl\`ag relaxed controlled solution to $(G,K,\nu)$, and defining
$\hat{y}_{t}^{\pi}\deq\int_{0}^{t}\E^{-\alpha s}\int c(x_{s}^{\pi},a,s)\pi_{s}(\D a)\,\D s$
for all $t\in\TT$, we have that of $y^{\pi}$ and $\hat{y}^{\pi}$
are indistinguishable. Conversely, if $\pi$ is a c\`adl\`ag relaxed controlled
solution to the problem $(G,K,\nu)$, and $y^{\pi}$ the corresponding
running costs process, then $(\Omega^{\pi},\Sigma^{\pi},\Ff^{\pi},\PP^{\pi},(x^{\pi},y^{\pi}),\pi)\in\DdD(H,K,\upsilon)$.
\end{prop}

Proof of Proposition~\ref{prop:augmentation-equivalence} is given in Appendix~\ref{subsec:proofs-s4}.
For the proof of Theorem~\ref{thm:d-f-equivalence}, we rely on the
results of \cite{Kurtz2001}.
\begin{thm}
\label{thm:ks}\cite[Theorem 1.11, Corollary 1.12]{Kurtz2001} Let
$(A,\nu_{0})$ be a relaxed controlled martingale problem, and suppose
$(\UU,\AA,A)$ and $\psi$ satisfy Assumption~\ref{cond:hasmarkov}.
If $\mu\in\FfF(A,\nu_{0})$ is a solution of the forward equation satisfying
\begin{gather}
\int_{0}^{t}\int_{\UU\times\AA}\psi(u,a,s)\,\mu_{s}(\D u\times\D a)\,\D s<\infty\quad\forall t\in\TT,\label{eq:psifin}
\end{gather}
then there exists a relaxed controlled solution $(u_{t}^{\pi},\pi_{t})_{t\in\TT}\in\RrR(A,\nu_{0})$
such that $\lL(u_{t})=\mu_{t}^{\UU}$ and $\pi_{t}=\mu_{t}^{\AA\mid\UU}(\cdot\mid u_{t})$
for all $t\in\TT$.
\end{thm}

\begin{rem}
In the given reference, this result is stated as applying to uncontrolled
problems, in particular, to an uncontrolled generator $\hat{A}:C_{b}(\UU)\supset\dD(\hat{A})\to\rR(\hat{A})\subset M(\UU)$.
We note however that the generator $\hat{A}$ is constructed from
a controlled generator $A$ satisfying an equivalent of our Assumption~\ref{cond:hasmarkov}
by integrating it over a transition function $\eta\in\Pp(\AA\mid\UU)$.
Re-writing the theorem in terms of the controlled generator $A$ recovers
the result stated above. The utility of constructing an uncontrolled
generator in this way is in the fact that after integrating over a
control, the resulting generator needs to satisfy notably weaker conditions
than the original, controlled generator, in particular, the generator
$\hat{A}$ can have its range extend to discontinuous measurable functions.
\end{rem}

An essential step in applying Theorem~\ref{thm:ks} is showing that
Eq.~(\ref{eq:psifin}) holds. A similar condition was already shown to
be true for c\`adl\`ag relaxed controlled solution in
Proposition~\ref{prop:regular-means-regular}, and the following
proposition can be viewed as an analogue of that result for solutions
of the forward equation.
\begin{prop}
\label{prop:regular-forward-equation}Suppose $(A,K,\nu_{0})$ is
a regular relaxed controlled martingale problem, and $\mu\in\FfF(A,K,\nu_{0})$.
Then 
\begin{gather}
\int_{\UU\times\AA}\left(1+\psi_{\UU}(u)+\psi_{\AA}(a)\right)\mu_{t}(\D u\times\D a)\leq L_{\psi}\E^{\Lambda_{\psi}t}\qquad\forall t\in\TT,\label{eq:psi-gen-bound-1}
\end{gather}
where
$\Lambda_{\psi}\deq\Lambda_{1}+\text{\ensuremath{\Lambda}}_{\AA}$ and
$L_{\psi}\in\RR_{>0}$ is independent of $\mu$. Eq.~(\ref{eq:psifin})
holds, and $\mu^{\UU}\in C(\TT,\Pp(\UU))$. Moreover, for all
$\epsilon>0$ and $t\in\TT$, there exists $\delta_{\epsilon,t}>0$ such
that
\begin{gather}
\biggl|\int_{\UU}f(u)\mu_{t}^{\UU}(\D u)-\int_{\UU}f(u)\mu_{s}^{\UU}(\D u)\biggr|<a_{f}\epsilon\quad\forall|t-s|<\delta_{\epsilon,t},\,f\in\dD(A),\label{eq:psi-unicon}
\end{gather}
where $a_{f}$ is as in Assumption~\ref{cond:hasmarkov}(\emph{iv}) and
$\delta_{\epsilon,t}$ does not depend on $f$ or $\mu$.
\end{prop}

Equipped with the above results, we can move on to the proof of Theorem~\ref{thm:d-f-equivalence}.
\begin{proof}[Proof of Theorem~\ref{thm:d-f-equivalence}.]
 By Proposition~\ref{prop:augmentation-equivalence}, we may consider
c\`adl\`ag relaxed controlled solutions $\DdD(H,K,\upsilon)$ instead
of solutions $\DdD(G,K,\nu)$ together with their associated costs
processes.

(\emph{i}) Let then $\mu\in\FfF(H,K,\upsilon)$. It is straight-forward
to verify that $\mu^{\XX\times\AA}\deq(\mu_{t}^{\XX\times\AA})_{t\in\TT}$
is in $\FfF(G,K,\nu)$, and so by Proposition~\ref{prop:regular-forward-equation},
\begin{gather*}
\int_{0}^{t}\int_{\XX\times\AA}\psi(x,a,s)\,\mu_{s}^{\XX\times\AA}(\D x\times\D a)\,\D s<\infty\quad\forall t\in\TT,
\end{gather*}
and Theorem~\ref{thm:ks} yields a relaxed controlled solution $\pi\in\RrR(G,K,\nu)$.

By \cite[Remark 3.5]{Kurtz2001}, we can suppose the solution $\pi$
obtained from Theorem~\ref{thm:ks} has the form
$x_{t}^{\pi}=\Gamma^{-1}(z_{t}^{\pi})$, where $z^{\pi}$ is an adapted
c\`adl\`ag process, and $\hat{\XX}$ and $\Gamma:\XX\to\hat{\XX}$ are
defined
\begin{gather*}
\hat{\XX}\deq\left[-\Vert f_{1}\Vert,+\Vert f_{1}\Vert\right]\times\left[-\Vert f_{2}\Vert,+\Vert f_{2}\Vert\right]\times\cdots,\\
\Gamma(x)\deq(f_{1}(x),f_{2}(x),\ldots)\quad\forall x\in\XX,
\end{gather*}
with the functions $(f_{k})_{k\in\NN}$ being as in Assumption~\ref{cond:hasmarkov}(\emph{v}).
The space $\hat{\XX}$ is compact and the mapping $\Gamma$ is continuous
with a measurable inverse $\Gamma^{-1}:\Gamma(\XX)\to\XX$. Since
we require in Assumption~\ref{cond:hasmarkov}(\emph{iii}) that $\dD(G)$
strongly separates points, then by \cite[Lemma 1]{Blount2010}, $\Gamma^{-1}$
is continuous. It then follows that $x^{\pi}=(x_{t}^{\pi}=\Gamma^{-1}(z_{t}^{\pi}))$
is also c\`adl\`ag up to the first time $t$ such that $\lim_{s\uparrow t}z_{t}^{\pi}\notin\Gamma(\XX)$
for all $t\in\TT$. Clearly, $z^{\pi}$ exits the image of $\XX$
when $\lim_{s\uparrow t}\psi_{\XX}(x_{s}^{\pi})=\infty$. We can now
use the argument of Proposition~\ref{prop:regular-means-regular}
to estimate the first time at which the c\`adl\`ag process $x^{\pi}$
reaches infinity, and conclude that almost surely this never happens.
Therefore, the solution is c\`adl\`ag for all $t\in\TT$, almost surely.

For part (\emph{ii}), suppose $\pi\in\DdD(H,K,\upsilon)$ and let $\mu$
be as in Eq.~(\ref{eq:equiv-converse}). Letting $fg\in\dD(H)$ be
arbitrary, and taking the expectation of Eq.~(\ref{eq:martingale}),
one finds that $\mu\in\FfF(H,K,\upsilon)$.
\end{proof}
It remains for us to provide the short proof of Theorem~\ref{thm:equivalence}.
\begin{proof}[Proof of Theorem~\ref{thm:equivalence}.]
 The equality of $\pP_{\Ll}$ and $\pP_{\Pp}$-optimal values follows
now from Theorem~\ref{thm:d-f-equivalence} and the law invariance
of the risk function. In addition, if $\pi\in\DdD(G,K,\nu)$ is $\pP_{\Ll}$-optimal,
then there exists a $\pP_{\Pp}$-optimal $\mu\in\FfF(H,K,\upsilon)$
with the same optimal value. Again, by Theorem~\ref{thm:d-f-equivalence}
there exists a $\tilde{\pi}\in\DdD(G,K,\nu)$ constructed from $\mu$
such that the control process has the form given in the statement
of the theorem.
\end{proof}
For the proof of Theorem~\ref{thm:existence}, we introduce a family
of metrics, whose members each induce a topology at least as fine
that of weak convergence.
\begin{lem}
\label{lem:G-metric}Suppose $\UU$ is Polish, $\Gg\subset C_{b}(\UU)$
is an algebra that strongly separates points, and $\Vert\cdot\Vert_{\Gg}:\Gg\to\RR_{\geq0}$
is a seminorm. We define $d_{\Gg}:\Pp(\UU)\times\Pp(\UU)\to\RR_{\geq0}$
via
\begin{gather}
d_{\Gg}\left(\mu,\nu\right)\deq\sup\biggl\{\frac{\left|\int f(u)\mu(\D u)-\int f(u)\nu(\D u)\right|}{\Vert f\Vert+\Vert f\Vert_{\Gg}}\biggm|f\in\Gg,\,f\neq0\biggr\}\quad\forall\mu,\nu\in\Pp(\UU),\label{eq:G-metric-def}
\end{gather}
which is equivalent to the definition $d_{\Gg}(\mu,\nu)\deq\sup_{f\in\Gg_{1}}\bigl\{\bigl|\int f(u)\mu(\D u)-\int f(u)\nu(\D u)\bigr|\bigr\}$,
where $\Gg_{1}\deq\{f\in\Gg\mid\Vert f\Vert+\Vert f\Vert_{\Gg}\leq1\}$
for all $\mu,\nu\in\Pp(\UU)$.

(\emph{i}) The mapping $d_{\Gg}$ is a metric, and convergence in
the topology induced by $d_{\Gg}$ implies weak convergence. (\emph{ii})
If $\Gg^{\prime}$ is a subalgebra of $\Gg$, and there is a seminorm
$\Vert\cdot\Vert_{\Gg^{\prime}}:\Gg\to\RR_{\geq0}$ such that $\Vert\cdot\Vert_{\Gg}\leq\Vert\cdot\Vert_{\Gg^{\prime}}$,
then 
\begin{gather*}
d_{\Gg^{\prime}}\left(\mu,\nu\right)\leq d_{\Gg}\left(\mu,\nu\right)\quad\forall\mu,\nu\in\Pp(\UU).
\end{gather*}
\end{lem}

The $d_{\Gg}$-metrics defined above include the bounded Lipschitz
metric, $d_{bl}$ defined below, as a special case. As a consequence
of Lemma~\ref{lem:G-metric}, we obtain the following comparison
result.
\begin{cor}
\label{cor:dbl-comparison}Suppose $\UU$ is Polish, and define the
bounded Lipschitz metric $d_{bl}:\Pp(\UU)\times\Pp(\UU)\to\RR_{\geq0}$
as in \cite[Section 11.3]{Dudley2002},
\begin{gather*}
d_{bl}\left(\mu,\nu\right)\deq\sup\biggl\{\left|\int f(u)\mu(\D u)-\int f(u)\nu(\D u)\right|\biggm|f\in C_{bl}(\UU),\,\Vert f\Vert_{bl}\leq1\biggr\},
\end{gather*}
for all $\mu,\nu\in\Pp(\UU)$. Then, if $(\Gg,\Vert\cdot\Vert_{\Gg})$
is as in the statement Lemma~\ref{lem:G-metric}, and $\Gg\subset
C_{bl}(\UU)$ and $\Vert\cdot\Vert_{\Gg}\geq\Vert\cdot\Vert_{l}$ on
$\Gg$, then $d_{\Gg}(\mu,\nu)\leq d_{bl}(\mu,\nu)$ for all
$\mu,\nu\in\Pp(\UU)$ and the topology induced by $d_{\Gg}$ is
equivalent to the topology of weak convergence.
\end{cor}

\begin{proof}
The set $C_{bl}(\UU)$ is an algebra that strongly separates points,
and the Lipschitz constant $\Vert\cdot\Vert_{l}$ is a seminorm on
$C_{bl}(\UU)$. Hence the definition of $d_{bl}$ is a special case
of the metrics defined in Lemma~\ref{lem:G-metric}. From there,
it follows that $d_{\Gg}(\mu,\nu)\leq d_{bl}(\mu,\nu)$ for all $\mu,\nu\in\Pp(\UU)$,
and therefore the topology induced by $d_{bl}$ is finer than that
of $d_{\Gg}$, which itself is finer than the topology of weak convergence.
However, by \cite[Theorem 11.3.3]{Dudley2002}, $d_{bl}$ and $d_{P}$
yield equivalent topologies, and so also $d_{\Gg}$ induces the same
topology.
\end{proof}
Note that tightness is a topological property: If $\UU$ is Polish,
then the tightness of a set $M\subset\Pp(\UU)$ implies convergence in
any metric that is equivalent to the Prokhorov metric $d_{P}$.  Since
sequential compactness implies compactness in metric spaces, tightness
of $M$ further implies it has compact closure, regardless of which
(equivalent) metric is used. Continuing this line of reasoning, we get
the following result.
\begin{cor}
\label{cor:dG-complete}Suppose the assumptions of Corollary~\ref{cor:dbl-comparison}
hold. Then the metric $d_{\Gg}$ is complete.
\end{cor}

\begin{proof}
We may simply follow the proof of \cite[Theorem 11.5.4]{Dudley2002};
the details are omitted here. As argued above, tightness of a
$M\subset\Pp(\UU)$ implies that $M$ has compact closure in any metric
equivalent to $d_{P}$. Hence, $M$ is totally bounded in any of the
metrics $d_{bl}$, $d_{\Gg}$, or $d_{P}$. Total boundedness of a
$M\subset\Pp(\UU)$ can subsequently be shown to imply tightness, and
hence convergence of a subsequence in any of the metrics. It then
suffices to note, as in \cite[Corollary 11.5.5]{Dudley2002} that
Cauchy sequences are totally bounded.
\end{proof}
\begin{rem}
In the light of the above discussion, it is clear that in Assumption~\ref{assu:existence}(\emph{i})
and Corollary~\ref{cor:dbl-comparison}, $(C_{bl}(\UU),\Vert\cdot\Vert_{l})$
can be replaced by any $(\Gg^{\ast},\Vert\cdot\Vert_{\Gg^{\ast}})$
satisfying assumptions of Lemma~\ref{lem:G-metric}, if $d_{bl}\leq d_{\Gg^{\ast}}$
and $d_{\Gg^{\ast}}$ is topologically equivalent to $d_{P}$. Note
also that at one extreme, we can choose $\Gg^{\ast\ast}=C_{b}(\UU)$
and $\Vert\cdot\Vert_{\Gg^{\ast\ast}}=0$. We then obtain a metric
$d_{\Gg^{\ast\ast}}$ such that $d_{\Gg}\leq d_{\Gg^{\ast\ast}}$
for all other metrics $d_{\Gg}$ obtained from Lemma~\ref{lem:G-metric}.
The metric $d_{\Gg^{\ast\ast}}$ coincides with the total variation
norm of signed measures, see {e.g.} \cite[Section 29]{Halmos1950},
and it is therefore not topologically equivalent to $d_{P}$.
\end{rem}

We shall also need the following basic statement regarding sets that
strongly separate points. We omit the proof, as it is a straight-forward
application of \cite[Lemma 4]{Blount2010}.
\begin{prop}
\label{prop:product-ssp}Let $\UU$ and $\VV$ be Polish spaces, and
suppose $\Gg\subset C_{b}(\UU)$ and $\Hh\subset C_{b}(\VV)$ strongly
separate points. Then $\Jj\deq\{fg\mid f\in\Gg,g\in\Hh\}\subset C_{b}(\UU\times\VV)$
strongly separates points.
\end{prop}

As the first step towards proving Theorem~\ref{thm:existence}, we
give a compactness result for families of solutions to the forward
equation.
\begin{lem}
\label{lem:fatso}Suppose Assumption~\ref{assu:existence} holds,
and $\{\mu^{(n)}\}_{n\in\NN}\subset\FfF(H,K,\upsilon)$ is such that
there are $Y\geq0$ and $N\in\NN$ for which
\begin{gather*}
\sup_{n>N}\limsup_{t\to\infty}\int y^{p}\mu_{t}^{(n)\,\YY}\leq Y\quad(\TT=\RR_{\geq0})\quad\text{or}\\
\sup_{n>N}\int y^{p}\mu_{T}^{(n)\,\YY}\leq Y\quad(\TT=[0,T]).
\end{gather*}
Then there exists a sequence $(n_{k})_{k\in\NN}$ and a $\mu\in M(\TT,\Pp(\XX\times\YY\times\AA))$
such that $\mu^{(n_{k})}\warrow\mu$, $\mu^{\XX\times\YY}\in C(\TT,\Pp(\XX\times\YY))$,
and $\mu_{t}^{(n_{k})\,\XX\times\YY}\Rightarrow\mu_{t}^{\XX\times\YY}$
for all $t\in\TT$ and the limit $\mu$ satisfies the forward equation,
$\mu\in\FfF(H,K,\upsilon)$.
\end{lem}

\begin{proof}
We first note that the sets $\{\mu_{t}^{(n)}\}_{n\in\NN}$ are tight
for each $t\in\TT$. The functions $\{\mu^{(n)\,\XX\times\AA}\}_{n\in\NN}$
are solutions to the forward equation corresponding to $G$, and the
assumptions of Proposition~\ref{prop:regular-forward-equation} hold.
By Eq.~(\ref{eq:psi-gen-bound-1}), inf-compactness of $\psi_{\XX}$
and $\psi_{\AA}$, and Proposition~\ref{prop:tightness}(\emph{iii}),
for every $t\in\TT$ the set of measures $\{\mu_{t}^{(n)\,\XX\times\AA}\}_{n\in\NN}$
is tight. From the forward equation for $H$, by considering non-negative
functions $g_{\ell}\in\dD(G)$, $\ell\in\NN$, that are constant on
$\XX$ and increasing towards $(\cdot)^{p}$, and using the monotone
convergence theorem, we find 
\begin{gather*}
\int_{\YY}y^{p}\mu_{t}^{(n)\,\YY}(\D y)=\int_{0}^{t}\E^{-\alpha s}\int_{\XX\times\YY\times\AA}c(x,a,s)py^{p-1}\mu_{t}^{(n)}(\D x\times\D y\times\D a)\,\D s\quad\forall t\in\TT,\,n>N.
\end{gather*}
Therefore, $t\to\int y^{p}\mu_{t}^{(n)\,\YY}(\D y)$ is increasing
and bounded by $Y$ for all $t\in\TT$ for at least all $n>N$. Proposition~\ref{prop:tightness}(\emph{ii})
asserts the tightness of $\{\mu_{t}^{(n)\,\YY}\}_{n\in\NN}$, so that
$\{\mu_{t}^{(n)}\}_{n\in\NN}$ is tight for all $t\in\TT$.

We next consider the continuity of the solutions $\{\mu^{(n)}\}_{n\in\NN}$.
Let $\Hh$ be the linear span of $\dD(H)$ and define
\begin{gather*}
\Vert h\Vert_{\Hh}\deq\sup_{y\in\YY}a_{h(\cdot,y)}+\Vert\partial_{y}h\Vert\quad\forall h\in\Hh,
\end{gather*}
where $\partial_{y}$ is the partial derivative along the $\YY$-space.
By using the assumption that $a_{\cdot}$ is a seminorm, it follows
that $\Vert\cdot\Vert_{\Hh}$ is also a seminorm. The set $\Hh$ is
closed under multiplications, is therefore an algebra, and by Proposition~\ref{prop:product-ssp}
strongly separates points. Moreover, $\Vert h\Vert_{\Hh}\geq\Vert h\Vert_{l}$
for all $h\in\Hh$, which can be shown using elementary estimates:
For any $h\in\Hh$,
\begin{align*}
\Vert h\Vert_{\Hh} & \geq\sup_{y\in\YY}\Vert h(\cdot,y)\Vert_{l}+\sup_{(x,y)\in\XX\times\YY}\left|\partial_{y}h(x,y)\right|\\
 & \geq\sup_{y\in\YY}\Vert h(\cdot,y)\Vert_{l}+\sup_{x\in\XX}\Vert h(x,\cdot)\Vert_{l}\\
 & =\sup_{y\in\YY}\sup_{x^{\prime}\neq x}\frac{\left|h(x^{\prime},y)-h(x,y)\right|}{d\left(x^{\prime},x\right)}+\sup_{x\in\XX}\sup_{y^{\prime}\neq y}\frac{\left|h(x,y^{\prime})-h(x,y)\right|}{\left|y^{\prime}-y\right|},\\
\Vert h\Vert_{l} & =\sup_{(x,y)\neq(x^{\prime},y^{\prime})}\frac{\left|h(x^{\prime},y^{\prime})-h(x,y)\right|}{d\left(x^{\prime},x\right)\vee\left|y^{\prime}-y\right|}\\
 & \leq\sup_{(x,y)\neq(x^{\prime},y^{\prime})}\frac{\left|h(x^{\prime},y)-h(x,y)\right|}{d\left(x^{\prime},x\right)\vee\left|y^{\prime}-y\right|}+\sup_{(x,y)\neq(x^{\prime},y^{\prime})}\frac{\left|h(x^{\prime},y^{\prime})-h(x^{\prime},y)\right|}{d\left(x^{\prime},x\right)\vee\left|y^{\prime}-y\right|}\\
 & \leq\sup_{x\neq x^{\prime}}\sup_{y^{\prime\prime}\in\YY}\frac{\left|h(x^{\prime},y^{\prime\prime})-h(x,y^{\prime\prime})\right|}{d\left(x^{\prime},x\right)}+\sup_{y\neq y^{\prime}}\sup_{x^{\prime\prime}\in\XX}\frac{\left|h(x^{\prime\prime},y^{\prime})-h(x^{\prime\prime},y)\right|}{\left|y^{\prime}-y\right|}.
\end{align*}
From this, we get $\Vert h\Vert_{\Hh}\geq\Vert h\Vert_{l}$ for all
$h\in\Hh$. Selecting $\Gg=\Hh$ in Lemma~\ref{lem:G-metric}, we
obtain a metric $d_{\Hh}$ on $\Pp(\XX\times\YY)$, and by Corollary~\ref{cor:dbl-comparison},
$d_{\Hh}$ induces the topology of weak convergence.

We now use the metric $d_{\Hh}$ to estimate the distances between
$\mu_{t}^{(n)\,\XX\times\YY}$ and $\mu_{s}^{(n)\,\XX\times\YY}$,
$n\in\NN$ and $t,s\in\TT$. Proposition~\ref{prop:regular-forward-equation}
is not directly applicable to the augmented problem $(H,K,\upsilon)$,
as this has not been established to be regular, but an analogue of
Eq.~(\ref{eq:psi-unicon}) nonetheless holds. A straight-forward
calculation yields that
\begin{align*}
\biggl|\int h(x,y)&\mu_{t}^{(n)\,\XX\times\YY}(\D x\times\D y) - \int h(x,y)\mu_{s}^{(n)\,\XX\times\YY}(\D x\times\D y)\biggr|\\
 & \leq\sup_{y\in\YY}a_{h(\cdot,y)}\int_{s}^{t}\int\psi(x,a,r)\mu_{r}(\D x\times\D y\times\D a)\,\D r\\
 & \qquad+\Vert\partial_{y}h\Vert\int_{s}^{t}\E^{-\alpha r}\int c(x,a,r)\mu_{r}(\D x\times\D y\times\D a)\,\D r\\
 & \leq\sup_{y\in\YY}a_{h(\cdot,y)}L_{1}^{1/\beta_{1}}\int_{s}^{t}\int\left(1+\psi_{\UU}(u)+\psi_{\AA}(a)\right)^{1/\beta_{1}}\mu_{r}(\D x\times\D y\times\D a)\,\D r\\
 & \qquad+\Vert\partial_{y}h\Vert L_{c}^{1/\beta_{c}}\int_{s}^{t}\E^{-\alpha r}\int\left(1+\psi_{\UU}(u)+\psi_{\AA}(a)\right)^{1/\beta_{c}}\mu_{r}(\D x\times\D y\times\D a)\,\D r\\
 & \leq\Vert h\Vert_{\Hh}\left(L_{1}^{1/\beta_{1}}\vee L_{c}^{1/\beta_{c}}\right)\int_{s}^{t}\int\left(1+\psi_{\UU}(u)+\psi_{\AA}(a)\right)\mu_{r}(\D x\times\D y\times\D a)\,\D r,
\end{align*}
for all $n\in\NN$, $t,s\in\TT$ and $h\in\Hh$. Estimating as in
Eq.~(\ref{eq:f-bound}), we find that for all $\epsilon>0$ and $t\in\TT$,
there exists a $\delta_{\epsilon,t}>0$ such that
\begin{gather}
\sup_{n\in\NN}\biggl|\int h(x,y)\mu_{t}^{(n)\,\XX\times\YY}(\D x\times\D y)-\int h(x,y)\mu_{s}^{(n)\,\XX\times\YY}(\D x\times\D y)\biggr|\leq\Vert h\Vert_{\Hh}\epsilon,\label{eq:eqc-estim-1}
\end{gather}
for all $|t-s|<\delta_{\epsilon,t}$ and $h\in\Hh$. From the definition
of $d_{\Hh}$ and Eq.~(\ref{eq:eqc-estim-1}), we get
\begin{gather*}
d_{\Hh}\left(\mu_{t}^{(n)\,\XX\times\YY},\mu_{s}^{(n)\,\XX\times\YY}\right)\leq\epsilon\quad\forall|t-s|<\delta_{\epsilon,t}.
\end{gather*}
Therefore, the set $\{\mu^{(n)\,\XX\times\YY}\}_{n\in\NN}$ is pointwise
equicontinuous when considered as a family of mappings from $(\TT,|\cdot|)$
to $(\Pp(\XX\times\YY),d_{\Hh})$. Above, we already showed that $\{\mu_{t}^{(n)\,\XX\times\YY}\}_{n\in\NN}$
are tight for each $t\in\TT$, and hence have compact closure. By
the Arzel\`a-Ascoli theorem \cite[Theorem 4.17]{Kelley1975}, $\{\mu^{(n)\,\XX\times\YY}\}_{n\in\NN}$
has compact closure, and consequently there is a subsequence $(n_{k})_{k\in\NN}$
such that $(\mu^{(n_{k})\,\XX\times\YY})_{k\in\NN}\in C(\TT,(\Pp(\XX\times\YY),d_{\Hh}))^\NN$
converges to a limit $\mu^{\XX\times\YY}\in C(\TT,(\Pp(\XX\times\YY),d_{\Hh}))$.
Since $d_{\Hh}$ and $d_{P}$ are topologically equivalent, the limit
$\mu^{\XX\times\YY}\in C(\TT,(\Pp(\XX\times\YY),d_{P}))$ and $\mu_{t}^{(n_{k})\,\XX\times\YY}\Rightarrow\mu_{t}^{\XX\times\YY}$
for all $t\in\TT$. For simplicity, we suppose the whole sequence
converges.

We remark that albeit $d_{P}$ and $d_{\Hh}$ are topologically equivalent,
there appears to be no easy way of replacing the latter by the former
in the above argument. This is because equicontinuity depends on the
properties of the metric rather than that of the topology generated
by it. The same applies for uniform convergence.

Next, we want to show that $\mu^{(n)}\warrow\mu$. Let $L_{\psi}$,
$\Lambda_{\psi}$ be as in the statement of Proposition~\ref{prop:regular-forward-equation},
and define $\kappa^{(n)}\in\Pp(\XX\times\YY\times\AA\times\TT)$ via
\begin{align*}
\kappa^{(n)}(\D x\times\D y\times\D a\times\D t) & \deq(\Lambda_{\psi}+1)\E^{-(\Lambda_{\psi}+1)t}\Nn_{t}^{(n)\,-1}\Psi(x,a)\mu_{t}^{(n)}(\D x\times\D y\times\D a)\,\D t,\\
\Psi(x,a) & \deq\left(1+\psi_{\XX}(x)+\psi_{\AA}(a)\right)^{1/\beta_{1}}\quad\forall(x,a)\in\XX\times\AA,\\
\Nn_{t}^{(n)} & \deq\int\Psi(x,a)\mu_{t}^{(n)}(\D x\times\D y\times\D a)\\
 & \leq L_{\psi}^{1/\beta_{1}}\E^{\frac{\Lambda_{\psi}}{\beta_{1}}t}\quad\forall n\in\NN,\,t\in\TT.
\end{align*}
Since $\Psi\geq1$, $\Nn_{t}^{(n)}\leq1$ for all $n\in\NN$ and $t\in\TT$.
We prove that $\{\kappa^{(n)}\}_{n\in\NN}$ is tight, which we do
by showing that its marginals are tight, and use Proposition~\ref{prop:tightness}(\emph{i}).
The tightness of the time-marginals is trivial, since they are all
the same. For the $\XX\times\AA$-marginal, we estimate
\begin{align*}
\int\left(1+\psi_{\XX}(x)+\psi_{\AA}(a)\right)^{1-1/\beta_{1}}&\kappa^{(n)\,\XX\times\AA}(\D x\times\D a) = \int(\Lambda_{\psi}+1)\E^{-(\Lambda_{\psi}+1)t}\Nn_{t}^{(n)\,-1}\\
 & \qquad\times\int\left(1+\psi_{\XX}(x)+\psi_{\AA}(a)\right)\mu_{t}^{(n)}(\D x\times\D y\times\D a)\,\D t\\
 & \leq\int(\Lambda_{\psi}+1)\E^{-(\Lambda_{\psi}+1)t}L_{\psi}\E^{\Lambda_{\psi}t}\,\D t\\
 & \leq L_{\psi}(\Lambda_{\psi}+1),
\end{align*}
and so Proposition~\ref{prop:tightness}(\emph{iii}) implies that
$\{\kappa^{(n)\,\XX\times\AA}\}_{n\in\NN}$ is tight. For the $\YY$-marginal,
we use Young's inequality to get 
\begin{align*}
\int y^{\frac{p(\beta_{1}-1)}{\beta_{1}}}&\kappa^{(n)\,\YY}(\D y) \leq\int\int(\Lambda_{\psi}+1)\E^{-(\Lambda_{\psi}+1)t}y^{\frac{p(\beta_{1}-1)}{\beta_{1}}}\Psi(x,a)\mu_{t}^{(n)}(\D x\times\D y\times\D a)\,\D t\\
 & \leq\frac{1}{\beta_{1}}\int(\Lambda_{\psi}+1)\E^{-(\Lambda_{\psi}+1)t}\int\left[(\beta_{1}-1)y^{p}+\Psi(x,a)^{\beta_{1}}\right]\mu_{t}^{(n)}(\D x\times\D y\times\D a)\,\D t\\
 & =\frac{\beta_{1}-1}{\beta_{1}}\int(\Lambda_{\psi}+1)\E^{-(\Lambda_{\psi}+1)t}\int y^{p}\mu_{t}^{(n)}(\D x\times\D y\times\D a)\,\D t\\
 & \qquad+\frac{1}{\beta_{1}}\int(\Lambda_{\psi}+1)\E^{-(\Lambda_{\psi}+1)t}\int\left(1+\psi_{\XX}(x)+\psi_{\AA}(a)\right)\mu_{t}^{(n)}(\D x\times\D y\times\D a)\,\D t\\
 & \leq\frac{\beta_{1}-1}{\beta_{1}}\int(\Lambda_{\psi}+1)\E^{-(\Lambda_{\psi}+1)t}Y\,\D t+\frac{1}{\beta_{1}}\int(\Lambda_{\psi}+1)\E^{-(\Lambda_{\psi}+1)t}L_{\psi}\E^{\Lambda_{\psi}t}\,\D t\\
 & \leq\frac{(\beta_{1}-1)Y+(\Lambda_{\psi}+1)L_{\psi}}{\beta_{1}},
\end{align*}
for all $n>N$, and thus $\{\kappa^{(n)\,\YY}\}_{n\in\NN}$ is tight.
We conclude that $\{\kappa^{(n)}\}_{n\in\NN}$ is tight, and hence
contains a convergent subsequence with a limit $\kappa\in\Pp(\XX\times\YY\times\AA\times\TT)$.
Clearly, the $\TT$-marginal of $\kappa$ must be $(\Lambda_{\psi}+1)\E^{-(\Lambda_{\psi}+1)t}\D t$,
since again, this is the marginal of all $\kappa^{(n)}$, $n\in\NN$.
Moreover, as multiplication of measures by bounded continuous functions,
in particular by $\Psi^{-1}\leq1$, preserves weak convergence, we
find that the limit has the form 
\begin{gather*}
\kappa(\D x\times\D y\times\D a\times\D t)=(\Lambda_{\psi}+1)\E^{-(\Lambda_{\psi}+1)t}\Nn^{-1}\Psi(x,a)\mu_{t}(\D x\times\D y\times\D a)\,\D t,
\end{gather*}
where $\Nn$ is a normalization coefficient. From this, the convergence
$\mu^{(n)}\warrow\mu$, follows. Note that for this convergence result,
showing that $\{(\Lambda_{\psi}+1)\E^{-(\Lambda_{\psi}+1)t}\mu_{t}^{(n)}(\D x\times\D y\times\D a)\,\D t\}_{n\in\NN}$
is tight would have sufficed, however, in the following we need the
original definition of $\kappa^{(n)}$ that additionally features
the $\Psi$ coefficient. We again assume the whole sequence converges.

We can now show that the limit $\mu$ satisfies the forward equation.
First note that since $\mu_{t}^{(n)\,\XX\times\YY}\Rightarrow\mu_{t}^{\XX\times\YY}$
for all $t\in\TT$, we have that
\begin{gather*}
\lim_{n\to\infty}\int f(x)g(y)\mu_{t}^{(n)\,\XX\times\YY}(\D x\times\D y)=\int f(x)g(y)\mu_{t}^{\XX\times\YY}(\D x\times\D y)
\end{gather*}
for all $t\in\TT$ and $fg\in\dD(H)$. To show that
\begin{align*}
  \int_{0}^{t}\int Hfg(x,y,a,t)&\mu_{s}^{(n)}(\D x\times\D y\times\D a)\,\D s
  \\ & \to \int_{0}^{t}\int Hfg(x,y,a,t)\mu_{s}(\D x\times\D y\times\D a)\,\D s
\end{align*}
for all $fg\in\dD(H)$ and $t\in\TT$, it suffices now to show that
\begin{align*}
  \int\int h(t)Hfg(x,y,a,t)&\mu_{s}^{(n)}(\D x\times\D y\times\D a)\,\D s
  \\ & \to\int\int h(t)Hfg(x,y,a,t)\mu_{s}(\D x\times\D y\times\D a)\,\D s
\end{align*}
for all $h\in C_{c}(\TT)$. We note that for all $fg\in\dD(H)$ and
$h\in C_{c}(\TT)$, we have
\begin{align*}
  \int\int h(t)g(y)&Gf(x,a,t)\mu_{t}^{(n)}(\D x\times\D y\times\D a)\,\D t
  \\ & = \int\frac{h(t)}{(\Lambda_{\psi}+1)\E^{-(\Lambda_{\psi}+1)t}}\frac{g(y)Gf(x,a,t)}{\Psi(x,a,t)}\kappa^{(n)}(\D x\times\D y\times\D a\times\D t).
\end{align*}
Recalling that $\psi\leq L_{1}^{1/\beta_{1}}\Psi$, and $|Gf|\leq a_{f}\psi$
for all $f\in\dD(G)$, the integrand on the right-hand side is bounded,
and continuous, and by the weak convergence of the sequence $(\kappa^{(n)})_{n\in\NN},$
we have that
\begin{align*}
  \lim_{n\to\infty}\int\int &h(t)g(y)Gf(x,a,t) \mu_{t}^{(n)}(\D x\times\D y\times\D a)\,\D t
  \\ & = \int\int h(t)g(y)Gf(x,a,t)\mu_{t}(\D x\times\D y\times\D a)\,\D t.
\end{align*}
Finally, again for all $fg\in\dD(H)$ and $h\in C_{c}(\TT)$,
\begin{multline*}
\int\int h(t)\E^{-\alpha t}f(x)g^{\prime}(y)c(x,a,t)\mu_{t}^{(n)}(\D x\times\D y\times\D a)\,\D t\\
=\int\frac{h(t)\E^{-\alpha t}}{(\Lambda_{\psi}+1)\E^{-(\Lambda_{\psi}+1)t}}\frac{f(x)g^{\prime}(y)c(x,a,t)}{\Psi(x,a,t)}\kappa^{(n)}(\D x\times\D y\times\D a\times\D t).
\end{multline*}
As above, the integrand is bounded since by Assumption~\ref{assu:main},
$c\leq L_{c}^{1/\beta_{c}}\Psi^{\beta_{1}/\beta_{c}}$ and $\beta_{c}\geq\beta_{1}$,
and weak convergence implies
\begin{align*}
  \lim_{n\to\infty}\int\int h(t)\E^{-\alpha t}&f(x)g^{\prime}(y)c(x,a,t)\mu_{t}^{(n)}(\D x\times\D y\times\D a)\,\D t
  \\&=\int\int h(t)\E^{-\alpha t}f(x)g^{\prime}(y)c(x,a,t)\mu_{t}(\D x\times\D y\times\D a)\,\D t.
\end{align*}
Combining the above, we have that $\mu\in\FfF(H,K,\upsilon)$. The
constraints are satisfied, as we assume $K$ to be closed in the weak
topology.
\end{proof}
We can now give the proof of Theorem~\ref{thm:existence}.
\begin{proof}[Proof of Theorem~\ref{thm:existence}]
We give the proof for the case where Assumption~\ref{assu:existence}(\emph{iii}a)
holds; in the (b) case where the cost rate $c$ and terminal cost
$v$ are bounded we may take the space of costs $\YY$ to be compact,
and skip tightness arguments that are otherwise necessary.

Let $\rho^{\ast}\in\RR\cup\{\infty\}$ be the optimal value of the
problem. If $\rho^{\ast}$ is infinite, then every solution to the
forward equation is optimal, and we are done. We suppose then that
$\rho^{\ast}<\infty$, and let $(\mu^{(n)})\in\FfF(H,K,\upsilon)^{\NN}$
be a minimizing sequence.

We first show that the $p$th moments of the running costs are bounded.
If $\TT=\RR_{\geq0}$, for a sufficiently large $N\in\NN$, we have
that
\begin{gather*}
\limsup_{t\to\infty}\rho(\mu_{t}^{(n)\,\YY})\leq\rho^{\ast}+1\quad\forall n>N,
\end{gather*}
and, since $\rho$ is coercive, there is a $Y\geq0$ such that
\begin{gather*}
\limsup_{t\to\infty}\int y^{p}\mu_{t}^{(n)\,\YY}\leq Y\quad\forall n>N.
\end{gather*}
If instead $\TT=[0,T]$, then
\begin{gather*}
\int\left(y+v(x)\right)^{p}\mu_{T}^{(n)\,\XX\times\YY}\leq Y\quad\forall n>N.
\end{gather*}
The assumptions of Lemma~\ref{lem:fatso} now hold, and we obtain
a subsequence $(\mu^{(n_{k})})_{n_{k}\in\NN}\in M(\TT,\Pp(\XX\times\YY\times\AA))$
with a limit $\mu=\mu^{(\infty)}\in M(\TT,\Pp(\XX\times\YY\times\AA))$.
We shall, as usual, suppose the whole sequence converges.

As the next step, we prove that for each $\mu^{(n)}$, $n\in\NN\cup\{\infty\}$,
the cost marginal distributions become stationary when $\TT=\RR_{\geq0}$,
and that the infinite time limit is obtained uniformly, that is, at
rates independent of $n$. We use much the same methods as above when
proving that the $\XX\times\YY$-marginals are continuous. Here, we
need to only focus on the $\YY$-marginals. We define $\Cc\deq\{f+f_{0}\mid f\in C_{c}^{(1)}(\YY),\,f_{0}\in\RR\}$,
$\Vert f\Vert_{\Cc}\deq\Vert f^{\prime}\Vert$ for all $f\in\Cc$.
Assumptions of Lemma~\ref{lem:G-metric} and Corollaries~\ref{cor:dbl-comparison}
and~\ref{cor:dG-complete} hold, and we obtain a complete metric
$d_{\Cc}$ defined on $\Pp(\YY)$.

From the forward equation, for all $g\in C_{c}^{(1)}(\YY)$ and $n\in\NN\cup\{\infty\}$,
\begin{align*}
\biggl|\int g(y)&\mu_{t}^{(n)}(\D y)-\int g(y)\mu_{s}^{(n)}(\D y)\biggr| =\left|\int_{s}^{t}\E^{-\alpha r}\int g^{\prime}(y)c(x,a,r)\mu_{r}^{(n)}(\D x\times\D y\times\D a)\,\D r\right|\\
 & \leq\Vert g^{\prime}\Vert\left|\int_{s}^{t}\E^{-\alpha r}\int L_{c}^{1/\beta_{c}}\left(1+\psi_{\XX}(x)+\psi_{\AA}(a)\right)^{1/\beta_{c}}\mu_{r}^{(n)}(\D x\times\D y\times\D a)\,\D r\right|\\
 & \leq L_{c}^{1/\beta_{c}}\Vert g^{\prime}\Vert\int_{s}^{t}\E^{-\alpha r}\left(\int\left(1+\psi_{\XX}(x)+\psi_{\AA}(a)\right)\mu_{r}^{(n)}(\D x\times\D y\times\D a)\right)^{1/\beta_{c}}\,\D r\\
 & \leq L_{c}^{1/\beta_{c}}\Vert g^{\prime}\Vert\int_{s}^{t}\E^{-\alpha r}\left(L_{\psi}\E^{\Lambda_{\psi}r}\right)^{1/\beta_{c}}\,\D r\\
 & \leq L_{c}^{1/\beta_{c}}L_{\psi}^{1/\beta_{c}}\Vert g^{\prime}\Vert\int_{s}^{t}\E^{\left(\frac{\Lambda_{\psi}}{\beta_{c}}-\alpha\right)r}\,\D r\\
 & \leq L_{c}^{1/\beta_{c}}L_{\psi}^{1/\beta_{c}}\Vert g^{\prime}\Vert\left|\frac{\E^{\left(\frac{\Lambda_{\psi}}{\beta_{c}}-\alpha\right)t}-\E^{\left(\frac{\Lambda_{\psi}}{\beta_{c}}-\alpha\right)s}}{\frac{\Lambda_{\psi}}{\beta_{c}}-\alpha}\right|\\
 & \eqd\Vert g^{\prime}\Vert\eta(s,t).
\end{align*}
Recalling the bound on $\alpha$ given in Assumption~\ref{assu:existence},
we now have that $\eta(s,t)=\eta(t,s)$ and $\eta(t,t)=0$ for all
$s,t\in\TT$, and $\eta(t,t+h)\to0$ for all $h\geq0$ as $t\to\infty$.
Thus, for any $\epsilon>0$, we can find $T_{\epsilon}\in\TT$ so
that $\eta(t,s)\leq\epsilon$ for all $t,s\geq T_{\epsilon}$. This
implies that
\begin{gather*}
d_{\Cc}\left(\mu_{t}^{(n)\,\YY},\mu_{s}^{(n)\,\YY}\right)\leq\epsilon\quad\forall t,s\geq T_{\epsilon},
\end{gather*}
and so by using the completeness of $d_{\Cc}$, there is a $\mu_{\infty}^{(n)\,\YY}\in\Pp(\YY)$
such that $\mu_{t}^{(n)\,\YY}\Rightarrow\mu_{\infty}^{(n)\,\YY}\in\Pp(\YY)$
as $t\to\infty$. This is to say, the cost distributions converge
to stationary distributions $\mu_{\infty}^{(n)\,\YY}$, $n\in\NN\cup\{\infty\}$
as time tends to infinity. In addition, a similar argument as in the
proof of Lemma~\ref{lem:fatso} shows that $\mu^{(n)\,\YY}\ccarrow\mu^{(\infty)\,\YY}$
(with the metric $d_{\Cc}$ assigned to probability measures on $\Pp(\YY)$).

Let $\epsilon>0$ be arbitrary. For every $t\in\TT$ there now is
an $N_{\epsilon,t}\in\NN$ such that $d_{\Cc}(\mu_{t}^{(n)\,\YY},\mu_{t}^{(\infty)\,\YY})\leq\epsilon$
for all $n\geq N_{\epsilon,t}$. Therefore, for all $t\ge T_{\epsilon/3}$
and $n\geq N_{\epsilon/3,t}$, we have
\begin{align*}
d_{\Cc}\left(\mu_{\infty}^{(n)\,\YY},\mu_{\infty}^{(\infty)\,\YY}\right) & \leq d_{\Cc}\left(\mu_{\infty}^{(n)\,\YY},\mu_{t}^{(n)\,\YY}\right)+d_{\Cc}\left(\mu_{t}^{(n)\,\YY},\mu_{t}^{(\infty)\,\YY}\right)+d_{\Cc}\left(\mu_{t}^{(\infty)\,\YY},\mu_{\infty}^{(\infty)\,\YY}\right)\\
 & \leq \epsilon,
\end{align*}
which implies
$\mu_{\infty}^{(n)\,\YY}\Rightarrow\mu_{\infty}^{(\infty)\,\YY}$.
Since $t\to\int y^{p}\mu_{t}^{(n)\,\YY}(\D y)$ is bounded by $Y$ and
increasing for every $n\in\NN\cup\{\infty\}$, each of these function
converges for every $n\in\NN\cup\{\infty\}$ to a limit as
$t\to\infty$. Then, using Skorokhod's representation theorem
\cite[Theorem 3.1.8]{EK1986}, the uniform boundedness of the $p$th
moments, and dominated convergence theorem, we have that $\int
y^{p}\mu_{\infty}^{(n)\,\YY}(\D y)\to\int
y^{p}\mu_{\infty}^{(\infty)\,\YY}(\D y)$.  Convergence of the moments
together with weak convergence implies convergence in the
$p$-Wasserstein metric $W^{p}$ \cite[Theorem 6.9]{Villani2009}.

If $\TT=[0,T]$, we can directly use the convergence $\mu_{T}^{(n)\,\XX\times\YY}\Rightarrow\mu_{T}^{(\infty)\,\XX\times\YY}$
and the continuity of $\Theta$ to conclude using the continuous mapping
theorem \cite[Corollary 3.1.9]{EK1986} that $\mu_{T}^{(n)\,\XX\times\YY}\circ\Theta^{-1}\Rightarrow\mu_{T}^{(\infty)\,\XX\times\YY}\circ\Theta^{-1}$.
Since the $p$th moments of $(x,y)\to\Theta(x,y)=y+v(x)$ with respect
to $\mu_{T}^{(n)\,\XX\times\YY}$ are uniformly bounded over $n>N$,
the same argument as above shows that the sequence converges also
in the $p$-Wasserstein metric.

We can now complete the proof. For $\TT=\RR_{\geq0}$, by continuity
of $\tilde{\rho}:(\Pp^{p}(\RR),W^{p})\to(\RR,|\cdot|)$ and the asymptotic
time convergence of the cost distributions, for all $n\in\NN$,
\begin{gather*}
\limsup_{t\to\infty}\tilde{\rho}\left(\mu_{t}^{(n)\,\YY}\right)=\tilde{\rho}\left(\mu_{\infty}^{(n)\,\YY}\right)
\end{gather*}
and taking the limit $n\to\infty$,
\[
\rho^{\ast}=\lim_{n\to\infty}\limsup_{t\to\infty}\tilde{\rho}\left(\mu_{t}^{(n)\,\YY}\right)=\tilde{\rho}\left(\mu_{\infty}^{(\infty)\,\YY}\right),
\]
and similarly for the case of $\TT=[0,T]$, assuming lower semicontinuity
of $\tilde{\rho}$,
\begin{gather*}
\rho^{\ast}\geq\liminf_{n\to\infty}\tilde{\rho}\left(\mu_{T}^{(n)\,\XX\times\YY}\circ\Theta^{-1}\right)\geq\tilde{\rho}\left(\mu_{T}^{(\infty)\,\XX\times\YY}\circ\Theta^{-1}\right).
\end{gather*}
Therefore the limits are optimal, and the proof is complete. 
\end{proof}

\section{\label{sec:numer}Numerical example}

Here, as a proof of concept, we present and solve a particularly simple
risk-aware optimization problem. The state and action spaces are compact
and the processes one-dimensional, but we emphasize that our general
setup allows for non-compact state and action spaces, and infinite
dimensional state and action spaces.

We consider a follower problem on a circle: The setup consists of
a (uncontrolled) stochastic process $(b_{t})_{t\in T}$ (the target)
being pursued by a controlled process $(v_{t})_{t\in T}$ (the pursuer)
whose objective is to minimize the distance between itself and the
target by choosing the direction and speed at which the pursuer moves.
We assume the pursuer's cost rate is the sum of the distance between
it and the target (as we define later), and the velocity squared.
\global\long\def\blo{\underline{a}}%
\global\long\def\bhi{\bar{a}}%
\global\long\def\ybar{\bar{y}}%

To formalize the problem, let the state space $\XX=\RR/2\pi\NN$ (we
interpret $\XX$ as being formed from copies of the interval $[0,2\pi)$),
equipped with the metric $d(\theta,\phi)=[1-\cos(\theta-\phi)]^{1/2}$
for all $\theta,\phi\in\XX$, and the action space $\AA=[\blo,\bhi]$
where $-\infty<\blo<\bhi<\infty$. The pursuer's action is interpreted
as its velocity along the circle: The pursuer's position is represented
by a process $(\phi_{t})_{t\in T}$, $\phi_{t}\in\XX$ for all $t\in\TT$,
such that $\D\phi_{t}=a_{t}\D t$ where $a_{t}\in\AA$ is the action
at time $t\in\TT$. We take the target's position to be the process
$(\sigma w_{t})_{t\in\TT}$ where $(w_{t})_{t\in\TT}$ is a Wiener
process on $\RR$, mapped onto $\XX$, and $\sigma>0$ is a constant
representing the magnitude of randomness in its motion. Let the difference
between the pursuer's and target's positions $x_{t}=\phi_{t}-w_{t}$
be the state process $(x_{t})_{t\in T}$, $x_{t}\in\XX$ for all $t\in\TT$.
This follows the stochastic differential equation
\begin{gather*}
\D x_{t}=a_{t}\D t-\D w_{t},
\end{gather*}
with an initial value $x_{0}=0$. This process has the generator $G:C^{(2)}(\XX)\to C(\XX\times\AA)$,
such that for all $f\in\dD(G)$,
\begin{gather}
Gf(x,a)=a\frac{\partial f}{\partial x}(x)+\frac{1}{2}\sigma^{2}\frac{\partial^{2}f}{\partial x^{2}}(x).\label{eq:X_G}
\end{gather}
We choose the cost rate function 
\begin{gather}
c(x,a)=d(x,0)^{2}+\gamma a^{2}=1-\cos x+\gamma a^{2},\label{eq:X_c}
\end{gather}
for all $x,a\in\XX\times\AA$ and where $\gamma\geq0$ is a parameter
representing the magnitude of the pursuer's cost of moving. We consider
discounted costs with a discount rate $\alpha>0$. The space $\YY$
of values of accumulated costs is $\YY=[0,\ybar]$, where $\ybar\deq2+\gamma\bhi^{2}$.

As our risk function, we use the entropic risk measure: Let
$\theta\in\RR$, and define $\rho_{\theta}\in\Pp(\YY)\to\RR$ as
\begin{gather}
\rho_{\theta}(\lambda)\deq\begin{cases}
\frac{1}{\theta}\ln\left[\int_{\YY}\E^{\theta y}\lambda(\D y)\right], & \theta\neq0,\\
\int_{\YY}y\lambda(\D y), & \theta=0,
\end{cases}\label{eq:X_rho_prime}
\end{gather}
for all $\lambda\in\Pp(\YY)$. The parameter $\theta$ represents risk
preferences: Positive values translate to risk-averse objectives, and
the larger $\theta$ is, the more the risks are weighted in assessing
risk. Conversely, negative values of $\theta$ imply risk seeking
preferences. Note that for $\theta$ close to zero,
$\rho_{\theta}(\lambda)=\int_{\YY}[y+y^{2}/(2\theta)]\lambda(\D
y)+\mathcal{O}(\theta^{2})$, so that $\rho_{\theta}$ approximates
linear-quadratic costs as a special case. The risk function
$\rho_{\theta}$ is not convex, but since the logarithm is strictly
increasing, we can just as well consider the equivalent risk function
\begin{gather}
\hat{\rho}_{\theta}(\lambda)\deq\int_{\YY}\E^{\theta y}\lambda(\D y),\label{eq:X_rho}
\end{gather}
for all $\lambda\in\Pp(\YY),$ which is linear. We also restrict
$\theta$ to non-negative values.

Given the compactness of $\XX$ and $\AA$, Assumptions~\ref{assu:main}
and \ref{assu:existence} are readily verified.

For the numerical solution of the problem, we discretize the forward
equation using standard methods, cf. \cite[Chapter 2.2]{ctmdpbook},
so that the discretized equation corresponds to a forward equation
on a discrete space. As the discretized system corresponds to a finite
state continuous-time controlled Markov chain, the weak convergence
results of \cite[Chapter 10]{Kushner2001} apply. The time axis is
truncated to a maximum time of $T^{\ast}\in\TT$, and the $\XX$,
$\YY$, $\AA$, and $\TT$ axis are discretized to $n$ equidistant
samples each; the discretized spaces are labeled with an underscore
$n$. The details of the construction are omitted, and proving the
convergence of this approach in the risk-aware case is beyond the
scope of this paper. The objective is linear in the cost distribution,
and we may use linear programming methods to solve the problem.

\begin{figure}
\begin{centering}
\includegraphics[width=0.5\textwidth]{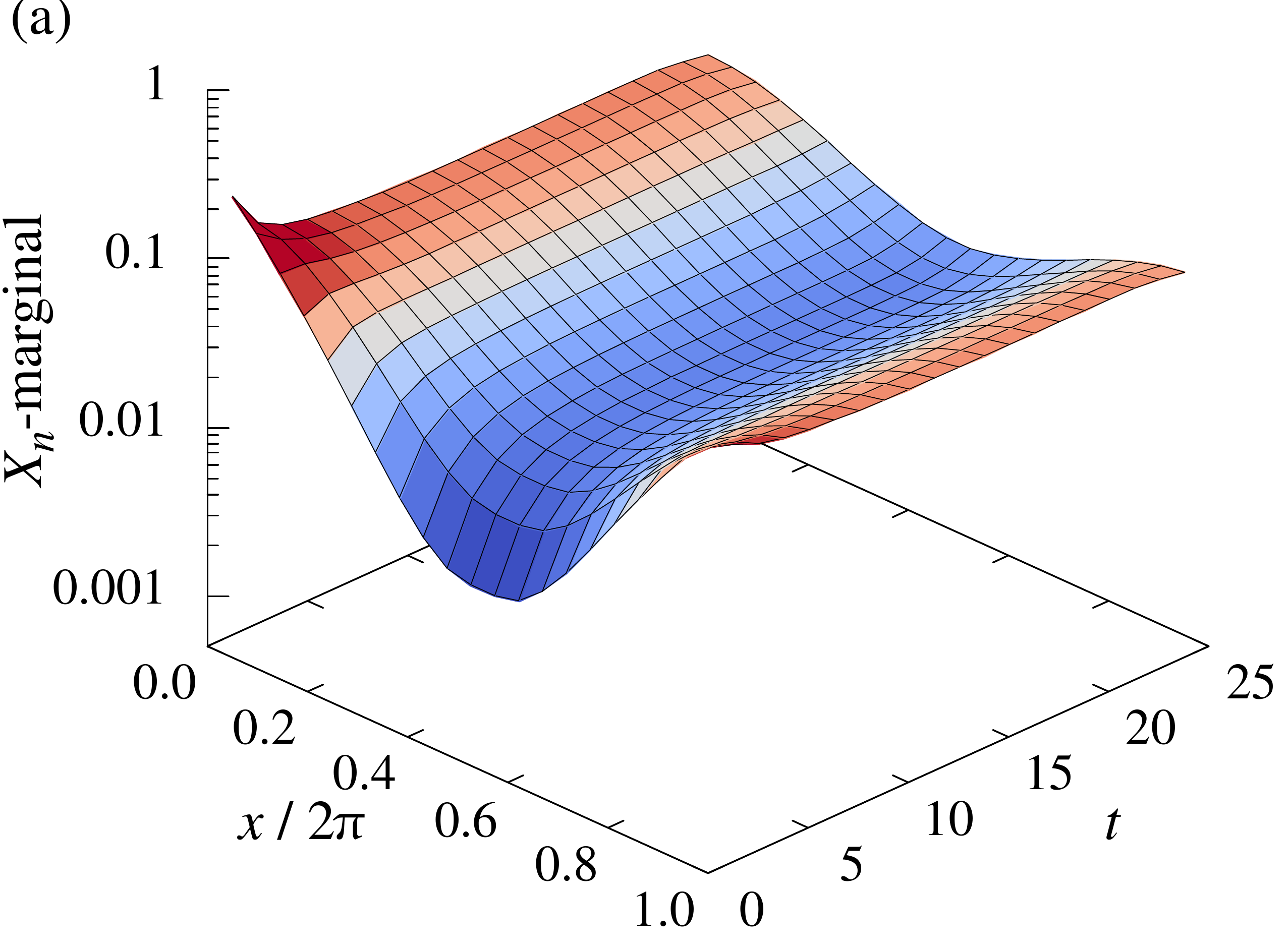}\includegraphics[width=0.5\textwidth]{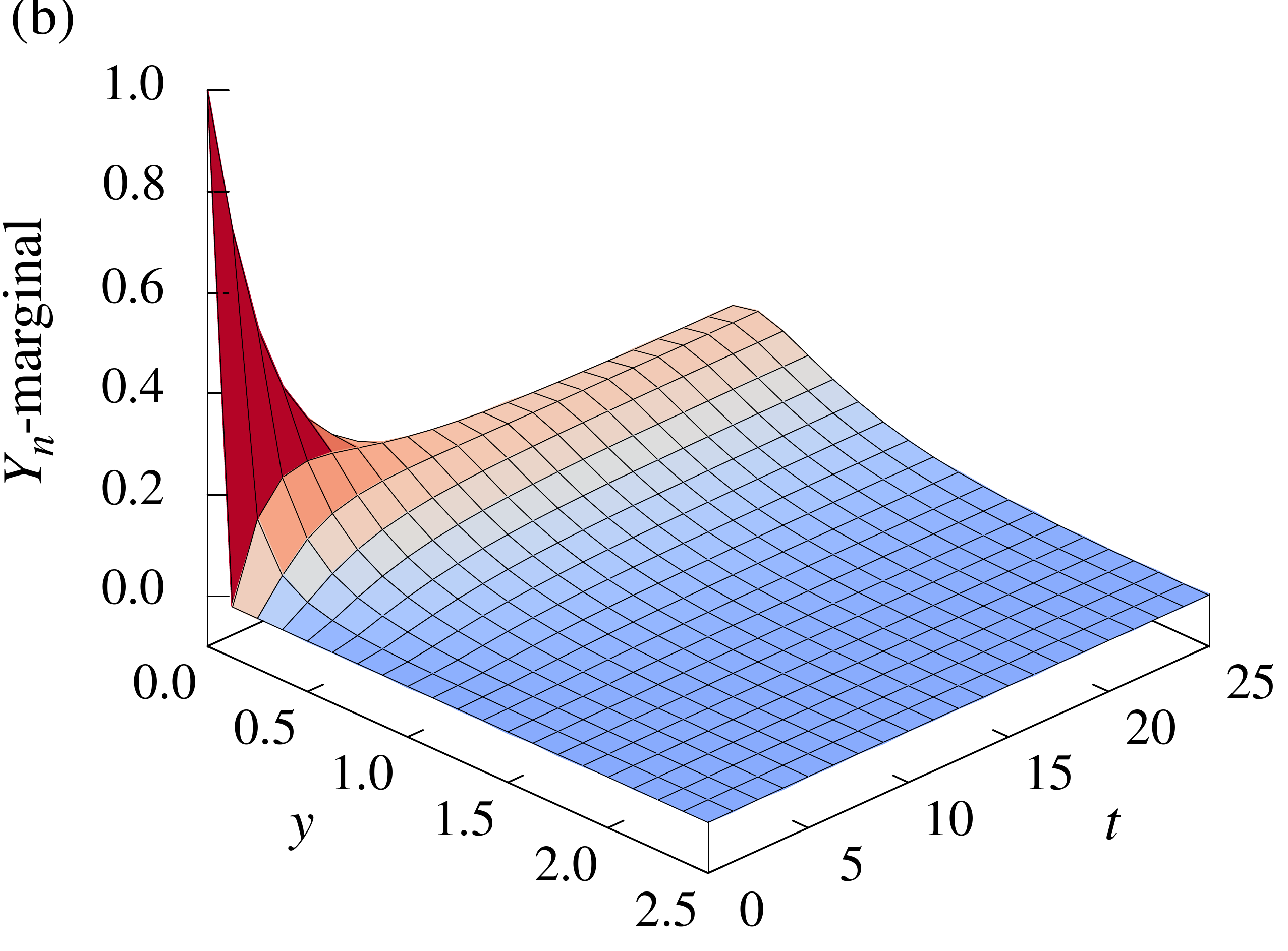}
\par\end{centering}
\caption{\label{fig:marginals}Marginal distributions obtained from an optimal
solution of the follower problem on a circle. Panel (\emph{a}): the
$\protect\XX_{n}$-marginal distribution as a function of time, logarithmic
scale is used to enhance visibility of features; (\emph{b}) the $\protect\YY_{n}$-marginal
distribution as a function of time. Both distributions are normalized
on the discrete space, with the coordinate axes showing the corresponding
$\protect\XX$ and $\protect\YY$ space values.}
\end{figure}

\begin{figure}
\centering{}\includegraphics[width=0.5\textwidth]{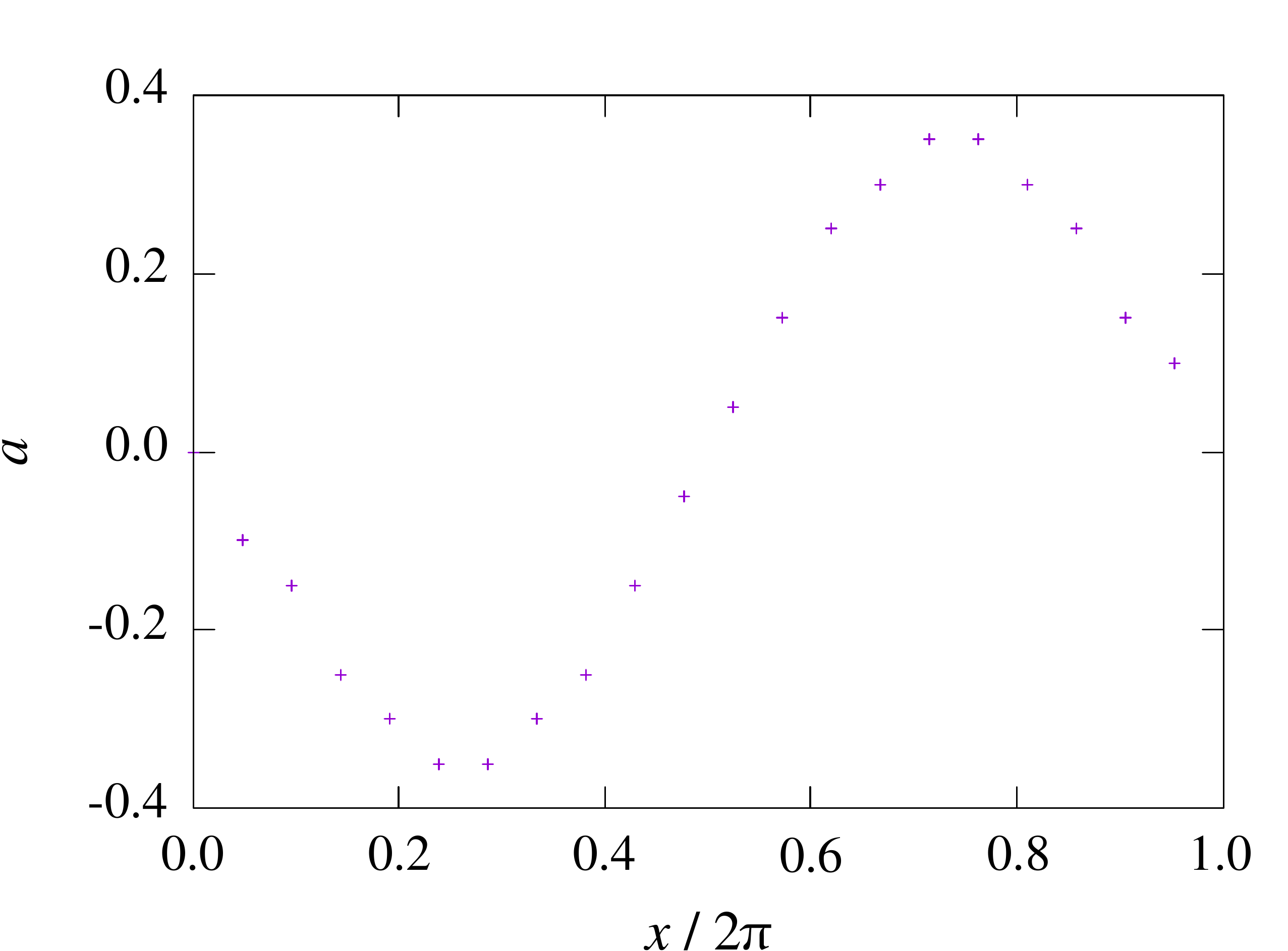}\caption{\label{fig:control}Visualization of the optimal control at time point
$T^{\ast}/2$, obtained by plotting non-zero points of the $\protect\XX_{n}\times\protect\XX_{n}$-marginal
distribution. As for each $x$ there appears only one action $a$
where the marginal distribution is zero, the control appears to be
strict.}
\end{figure}

We have numerically solved the problem for the following values of
the parameters:
\begin{gather*}
n=21,\,T^{\ast}=25,\\
\blo=-0.5,\,\bhi=0.5,\\
\sigma=1,\,\gamma=2,\,\alpha=0.25,\,\theta=1,\\
\nu=\delta_{0}.
\end{gather*}
The resulting linear programming problem was solved in parallel to
its dual using an interior point method described in \cite[Chapter 3]{Suvrit2012},
with a non-negativity condition imposed on the solution measure. An
optimal value of $\rho^{\ast}\approx1.79$ was found, and an effectively
zero duality gap was obtained. To visualize the solution, we plot
the $\XX_{n}$- and $\YY_{n}$-marginal distributions of the optimal
solution in Fig.~\ref{fig:marginals}. As expected, the $\XX_{n}$-marginal
distribution is centered around zero, meaning that the follower indeed
remains in the vicinity of its target, with the diffusion, or the
randomness of the target's motion creating the spreading of the distribution
as time advances. Similarly, the $\YY_{n}$-marginal distribution
spreads out after starting out concentrated at zero, becoming effectively
stationary well before the terminal time $T^{\ast}$. The optimal
control is depicted in Fig.~\ref{fig:control}, where we have arbitrarily
picked the time closest to $T_{n}^{\ast}/2$ to plot the non-zero
points of $\XX_{n}\times\AA_{n}$-marginal distribution. The control
appears to be strict, i.e. selecting a single action for a given state.
To further validate our results, we independently solved the risk-neutral
version of the problem using dynamic programming methods. Comparing
the results to the risk-aware problem with a small value of $\theta$,
we found very good agreement.

\section{Conclusions\label{sec:concl}}

We have presented a dynamic analytic formulation of a generic risk-aware
optimal control problem, with Polish, possibly infinite-dimesional
state and action spaces, and where the underlying dynamics are given
in the martingale formulation. Our primary goal was constructing a
practical method for solving risk-aware relaxed controlled martingale
problems, which we accomplished by providing an equivalent formulation
that takes the form of a nonlinear programming problem with a linear
constraint determining the controlled processes joint state and cost
distributions. The ability to obtain the full cost distribution is
characteristic to the dynamic analytic method, and is the reason why
it is well-suited for risk-aware problems: In this context, evaluating
the objective function requires knowledge of the joint state-cost
distribution. Contrast this to convex analytic methods, which only
yield the occupation measures which can be seen as one dimensional
projections of the joint state and cost distributions.

The dynamic analytic method was also capable of providing analytic
insight into the control problem. In particular, we found that the
optimal control processes can be taken to be Markov in time and the
system state and running costs. Significantly, we were also able to
prove the existence of optimal Markov controls under additional conditions.
We also provided a rather simple but instructive example of how the
method can be used to numerically solve risk-aware optimal control
problems.

\section*{Acknowledgements}
Acknowledgement. The authors would gratefully like to acknowledge
support for this research from the Singapore Ministry of Education,
under Tier 2 project MOE-2015-T2-2148.

\appendix

\section{Proofs and auxiliary results}

\subsection{Proofs for Section 2}
\begin{proof}[Proof of Proposition~\ref{prop:regular-means-regular}]
Let $\pi\in\DdD(A,K,\nu_{0})$, $K_{n}\deq\{u\in\UU\mid\psi_{\UU}(u)\leq n\}$
and $\tau_{n}\deq\inf\{t\in\TT\mid\psi_{\UU}(u_{t})\notin K_{n}\}$
for all $n\in\NN$. By the inf-compactness of $\psi_{\UU}$, $K_{n}$
is compact for all $n\in\NN$, so that $\tau_{n}$ is a stopping time
for all $n\in\NN$, see {e.g.} \cite[Proposition 2.1.5]{EK1986}. From
Eq.~(\ref{eq:martingale}), for all $k,n\in\NN$ and $t\in\TT$,
$s\leq t$,
\begin{align*}
m_{t\wedge\tau_{n}}^{\phi_{k}}-m_{s\wedge\tau_{n}}^{\phi_{k}} & =\phi_{k}(u_{t\wedge\tau_{n}}^{\pi})-\phi_{k}(u_{s\wedge\tau_{n}}^{\pi})-\int_{s}^{t\wedge\tau_{n}}\int_{\AA}A\phi_{k}(u_{r}^{\pi},a)\pi_{r}(\D a)\,\D r.
\end{align*}
By the optional stopping theorem \cite[Theorem 1.62]{Pardoux2014},
$\EE[m_{t\wedge\tau_{n}}^{\phi_{k}}-m_{s\wedge\tau_{n}}^{\phi_{k}}\mid\Ff_{s\wedge\tau_{n}}]=0$
for all $n,k\in\NN$, and using the boundedness of $A\phi_{k}$ by
$\psi_{\UU}$ and $\psi_{\AA}$,
\begin{align*}
\EE\Bigl[\phi_{k}(u_{t\wedge\tau_{n}}^{\pi})&\Bigm|\Ff_{s\wedge\tau_{n}}\Bigr] =\phi_{k}(u_{s\wedge\tau_{n}}^{\pi})+\EE\biggl[\int_{s}^{t\wedge\tau_{n}}\int_{\AA}A\phi_{k}(u_{r}^{\pi},a)\pi_{r}(\D a)\,\D r\biggm|\Ff_{s\wedge\tau_{n}}\biggr]\\
 & \leq\phi_{k}(u_{s\wedge\tau_{n}}^{\pi})+\EE\biggl[\int_{s}^{t}\Lambda_{1}\left(1+\psi_{\UU}(u_{r\wedge\tau_{n}}^{\pi})+\int_{\AA}\psi_{\AA}(a)\pi_{r}(\D a)\right)\,\D r\biggm|\Ff_{s\wedge\tau_{n}}\biggr]\\
 & \qquad\forall t\in\TT,\,s\leq t.
\end{align*}
Letting $k\to\infty$ and using the dominated convergence theorem
and Gr\"onwall's inequality \cite[Corollary 6.60]{Pardoux2014}, we
obtain
\begin{align}
  \EE\Bigl[\psi_{\UU}(u_{t\wedge\tau_{n}}^{\pi})\Bigm|\Ff_{s\wedge\tau_{n}}\Bigr]
  & \leq
  \biggl(\psi_{\UU}(u_{s\wedge\tau_{n}}^{\pi} + \Lambda_{1}(t-s)
  \notag \\
  &\qquad +\Lambda_{1}\EE\Bigl[\int_{s}^{t}\int_{\AA}\psi_{\AA}(a)\pi_{r}(\D a)\,\D r\Bigm|\Ff_{s\wedge\tau_{n}}\Bigr])\biggr)\E^{\Lambda_{1}(t-s)}, \label{eq:psi-u-estimate} 
\end{align}
for all $t\in\TT$, $s\leq t$. Setting
\begin{gather*}
b_{t}\deq\Lambda_{1}\EE\Bigl[t+\int_{0}^{t}\int_{\AA}\psi_{\AA}(a)\pi_{r}(\D a)\,\D r\Bigr],
\end{gather*}
from Eq.~(\ref{eq:psi-u-estimate}), we find that
\begin{align*}
\psi_{\UU}(u_{0}^{\pi}) & \geq\E^{-\Lambda_{1}t}\EE\Bigl[\psi_{\UU}(u_{t\wedge\tau_{n}}^{\pi})\Bigr]-b_{t}\\
 & \geq\E^{-\Lambda_{1}t}\EE\Bigl[\indic_{[\tau_{n},\infty)}(t)\psi_{\UU}(u_{t\wedge\tau_{n}}^{\pi})\Bigr]-b_{t}\\
 & \geq\E^{-\Lambda_{1}t}n\EE\Bigl[\indic_{[\tau_{n},\infty)}(t)\Bigr]-b_{t},
\end{align*}
and hence,
\begin{gather*}
\PP\bigl[\tau_{n}\leq t\bigr]\leq\E^{\Lambda_{1}t}\frac{\psi_{\UU}(u_{0}^{\pi})+b_{t}}{n}<\infty\quad\forall t\in\TT,\,n\in\NN,
\end{gather*}
where the finiteness follows from the regularity conditions given
in Definition~\ref{def:regular}(\emph{iv}, \emph{v}).

Defining $\tau_{\infty}\deq\limsup_{n\to\infty}\tau_{n}$, the time
for reaching infinity, we then get
\begin{gather*}
\PP\bigl[\tau_{\infty}\leq t\bigr]=\lim_{n\to\infty}\PP\bigl[\tau_{n}\leq t\bigr]=0\quad\forall t\in\TT,
\end{gather*}
and so the process $\psi_{\UU}(u_{\cdot})$ is almost surely finite
for all $t\in\TT$. The finiteness of $\int_{\AA}(1+\psi_{\UU}(u_{\cdot})+\psi_{\UU}(a))\pi_{\cdot}(\D a)$
is then a direct consequence of the above, and the regularity of $(A,K,\nu_{0})$.
\end{proof}
\begin{proof}[Proof of Proposition~\ref{prop:rho-rho-tilde-lsc}]
We prove the statement for the case of lower semicontinuity, for
continuity the argument is identical. Let $(\mu_{n})_{n\in\NN}\subset\Pp^{p}(\RR)$
be an arbitrary sequence converging to some $\mu\in\Pp^{p}(\RR)$
in the $p$-Wasserstein metric. This implies that $\mu_{n}\Rightarrow\mu$,
and $\int|x|^{p}\mu_{n}(\D x)\to\int|x|^{p}\mu(\D x)$ as $n\to\infty$
\cite[Theorem 6.9]{Villani2009}. By Skorokhod's representation theorem
\cite[Theorem 3.1.8]{EK1986}, there exists a probability space $(\tilde{\Omega},\tilde{\Sigma},\tilde{\PP})$
and random variables $(\tilde{X}_{n})_{n\in\NN}\subset\Ll(\tilde{\Omega};\RR)$,
$\tilde{X}\in\Ll(\tilde{\Omega};\RR)$ such that $\lL(\tilde{X}_{n})=\mu_{n}$
for all $n\in\NN$ and $\lL(\tilde{X})=\mu$, and $\tilde{X}_{n}\to\tilde{X}$
almost surely. Since $\mu_{n}\in\Pp^{p}(\RR)$, we also have that
$\tilde{X}_{n}\in\Ll^{p}(\tilde{\Omega};\RR)$ for all $n\in\NN$,
and similarly for $\tilde{X}$. An application of the dominated convergence
theorem shows that $\Vert\tilde{X}_{n}-\tilde{X}\Vert_{p}\to0$, and
so by the law invariance and lower semicontinuity of $\rho$, we have
that
\begin{gather*}
\liminf_{n\to\infty}\tilde{\rho}(\mu_{n})=\liminf_{n\to\infty}\rho(\tilde{X}_{n})\geq\rho(\tilde{X})=\tilde{\rho}(\mu),
\end{gather*}
and therefore $\tilde{\rho}$ is lower semicontinuous.
\end{proof}

\subsection{\label{subsec:proofs-s4}Proofs for Section 3}
\begin{proof}[Proof of Proposition~\ref{prop:augmentation-equivalence}.]
Let $\pi\in\DdD(H,K,\upsilon)$ and let $\hat{y}^{\pi}=(\hat{y}_{t}^{\pi})_{t\in\TT}$
as in the statement. It is straight-forward to see that $(\Omega^{\pi},\Sigma^{\pi},\Ff^{\pi},\PP^{\pi},x^{\pi},\pi)$
is a c\`adl\`ag relaxed controlled solution to $(G,K,\nu)$: It suffices
consider Eq.~(\ref{eq:martingale}) of the definition of a solution,
and select functions $f\in\dD(H)$ such that only depend on the $\XX$-component.
Regularity of $(G,K,\nu)$, Assumption~\ref{assu:main}, and Proposition~\ref{prop:regular-means-regular}
imply that $\psi(x_{\cdot}^{\pi})$, $\int_{\AA}\alpha\E^{-\alpha\cdot}c(x_{\cdot}^{\pi},a,\cdot)\pi_{\cdot}(\D a)$,
and $\hat{y}_{\cdot}^{\pi}$ are all almost surely finite for all
$t\in\TT$.

We first note that $y^{\pi}$ is continuous. Let $\dD_0 \deq \{g + g_0
\mid g \in C_c^{(1)}(\YY),\, g_0 \in \RR\}$.  By the martingale
property of the solutions,
\begin{gather*}
m_{t}^{g}\deq g(y_{t}^{\pi})-g(y_{s}^{\pi})-\int_{s}^{t}\int_{\AA}\E^{-\alpha r}c(x_{r}^{\pi},a,r)g^{\prime}(y_{r}^{\pi})\pi_{r}(\D a)\,\D r\quad\forall t\in\TT
\end{gather*}
is a martingale for all $g\in \dD_0$. Thus, for an arbitrary
$g\in \dD_0$ ,
\begin{gather*}
m_{t}^{g^{2}}=g(y_{t}^{\pi})^{2}-g(y_{s}^{\pi})^{2}-\int_{s}^{t}\int_{\AA}\E^{-\alpha r}c(x_{r}^{\pi},a,r)2g(y_{r}^{\pi})g^{\prime}(y_{r}^{\pi})\pi_{r}(\D a)\,\D s\quad\forall t\in\TT
\end{gather*}
is also a martingale. Using the above two equalities,
\begin{align*}
\left(g(y_{t}^{\pi})-g(y_{s}^{\pi})\right)^{2} & =g(y_{t}^{\pi})^{2}-2g(y_{t}^{\pi})g(y_{s}^{\pi})+g(y_{s}^{\pi})^{2}\\
 & =m_{t}^{g^{2}}-2g(y_{s}^{\pi})m_{t}^{g}+2\int_{s}^{t}\int_{\AA}\E^{-\alpha r}c(x_{r}^{\pi},a,r)g(y_{r}^{\pi})g^{\prime}(y_{r}^{\pi})\pi_{r}(\D a)\,\D r\\
 & \qquad-2g(y_{s}^{\pi})\int_{r}^{t}\int_{\AA}\E^{-\alpha r}c(x_{r}^{\pi},a,r)g^{\prime}(y_{r}^{\pi})\pi_{r}(\D a)\,\D r,
\end{align*}
so that
\begin{gather*}
\EE\left[\left(g(y_{t}^{\pi})-g(y_{s}^{\pi})\right)^{2}\right]=\EE\left[\left(\int_{s}^{t}\int_{\AA}\E^{-\alpha r}c(x_{r}^{\pi},a,r)\pi_{r}(\D a)2\left[g(y_{r}^{\pi})-g(y_{s}^{\pi})\right]g^{\prime}(y_{r}^{\pi})\,\D r\right)^{2}\right].
\end{gather*}
From this, in then follows that the (optional) quadratic variation
of $g(y_{\cdot}^{\pi})$ is almost surely zero: If $P_{t}=(t_{1},t_{2},\ldots,t_{n})$
is an arbitrary partition of $[0,t]$, $t\in\TT$, and $|P_{t}|\deq\max_{i\in\{1,\ldots,n-1\}}|t_{i+1}-t_{i}|$,
then
\begin{gather*}
\lim_{|P_{t}|\to0}\EE\left[\left(g(y_{t}^{\pi})-g(y_{s}^{\pi})\right)^{2}\right]=0\quad\forall t\in\TT,
\end{gather*}
at least almost surely for every $g \in \dD_0$. This implies that
the quadratic variation of $m^{g}$ is zero, and therefore $m^{g}$
is itself zero. So being, $y_{\cdot}^{\pi}$ is continuous, and
\begin{gather*}
g(y_{t}^{\pi})=g(y_{0}^{\pi})+\int_{0}^{t}\int_{\AA}\E^{-\alpha s}c(x_{s}^{\pi},a,s)\pi_{s}(\D a)g^{\prime}(y_{t}^{\pi})\,\D t,
\end{gather*}
for all $g \in \dD_0$, almost surely. Consider then any sequence of
functions $(g_{n})_{n\in\NN} \in \dD_0^\NN$ of the form $g_{n}(y)=y$
for all $y\leq n$, and for which $n \to g'_n(y)$ is non-decreasing for
all $y \in \YY$. From the above, it then follows using monotone
convergence theorem that
\begin{gather*}
y_{t}^{\pi}=\int_{0}^{t}\int_{\AA}\E^{-\alpha s}c(x_{s}^{\pi},a,s)\pi_{s}(\D a)\,\D s\quad\forall t\in\TT,
\end{gather*}
almost surely, and so $y^{\pi}$ is indistinguishable from $\hat{y}^{\pi}$.

For the converse part of the Proposition, let $\pi\in\DdD(G,K,\nu)$
and set $y^{\pi}$ to be the corresponding running costs, defined
as the integral in Eq.~(\ref{eq:running-costs}). Our goal is to
show that for all $fg\in\dD(H)$, the process
\begin{gather}
m_{t}^{fg}\deq f(x_{t}^{\pi})g(y_{t}^{\pi})-f(x_{0}^{\pi})g(0)-\int_{0}^{t}\int_{\AA}Hfg(x_{s}^{\pi},y_{s}^{\pi},a,s)\pi_{s}(\D a)\,\D s\quad\forall t\in\TT,\label{eq:m-equiv-1}
\end{gather}
is a martingale, which is sufficient to establish that $(\Omega^{\pi},\Sigma^{\pi},\Ff^{\pi},\PP^{\pi},(x^{\pi},y^{\pi}),\pi)\in\DdD(H,K,\upsilon)$.
The proof of this follows closely that of \cite[Lemma 4.3.4(a)]{EK1986},
however, the problem here does not quite satisfy the boundedness conditions
of that result.

Suppose $fg\in\dD(H)$ is arbitrary. Note first that $\EE[|m_{t}^{fg}|]<\infty$
for all $t\in\TT$; this follows from the bounds on $Gf$ and the
cost rate $c$ given in Assumption~\ref{assu:main}. Let $s,t\in\TT$,
$s\leq t$, and let $(t_{i})_{i\in\{1,\ldots,n\}}$ be an arbitrary
partition of $[s,t]$, $t_{1}=s$, $t_{n}=t$. Then, using the martingale
property for the $x^{\pi}$ process and the differentiability of $g$,
\begin{gather*}
m_{t,s}^{f}=f(x_{t}^{\pi})-f(x_{s}^{\pi})-\int_{s}^{t}\int_{\AA}Gf(x_{r}^{\pi},a,r)\pi_{r}(\D a)\,\D r,\\
0=g(y_{t}^{\pi})-g(y_{s}^{\pi})-\int_{s}^{t}\E^{-\alpha r}\int_{\AA}c(x_{r}^{\pi},a,r)g^{\prime}(y_{r}^{\pi})\pi_{r}(\D a)\,\D r,
\end{gather*}
$m_{\cdot,s}^{f}$ is a martingale, and defining for brevity, for
all $t^{\prime}\in[s,t]$,
\begin{gather*}
W_{t^{\prime}}\deq\int_{\AA}Gf(x_{t^{\prime}}^{\pi},a,t^{\prime})\pi_{t^{\prime}}(\D a),\\
V_{t^{\prime}}\deq\E^{-\alpha t^{\prime}}\int_{\AA}c(x_{t^{\prime}}^{\pi},a,t^{\prime})g^{\prime}(y_{t^{\prime}}^{\pi})\pi_{t^{\prime}}(\D a),
\end{gather*}
we find
\begin{align*}
\EE\Bigl[f(x_{t}^{\pi})g(y_{t}^{\pi})-f(x_{s}^{\pi})g(y_{s}^{\pi})\Bigm|\Ff_{s}\Bigr] & =\sum_{k=1}^{n}\EE\Bigl[f(x_{t_{k+1}}^{\pi})g(y_{t_{k+1}}^{\pi})-f(x_{t_{k}}^{\pi})g(y_{t_{k}}^{\pi})\Bigm|\Ff_{s}\Bigr]\\
 & =\sum_{k=1}^{n}\EE\Bigl[f(x_{t_{k+1}}^{\pi})\Bigl(g(y_{t_{k+1}}^{\pi})-g(y_{t_{k}}^{\pi})\Bigr)\\
 & \qquad\qquad+\Bigl(f(x_{t_{k+1}}^{\pi})-f(x_{t_{k}}^{\pi})\Bigr)g(y_{t_{k}}^{\pi})\Bigm|\Ff_{s}\Bigr]\\
 & =\sum_{k=1}^{n}\EE\biggl[\int_{t_{k}}^{t_{k+1}}\Bigl[f(x_{t_{k+1}}^{\pi})V_{r}+g(y_{t_{k}}^{\pi})W_{r}\Bigr]\,\D r\biggm|\Ff_{s}\biggr]\\
 & =\sum_{k=1}^{n}\EE\biggl[\int_{t_{k}}^{t_{k+1}}\biggl\{\Bigl[f(x_{r}^{\pi})+f(x_{t_{k+1}}^{\pi})-f(x_{r}^{\pi})\Bigr]V_{r}\\
 & \qquad\qquad+\Bigl[g(y_{r}^{\pi})+g(y_{t_{k}}^{\pi})-g(y_{r}^{\pi})\Bigr]W_{r}\biggr\}\,\D r\biggm|\Ff_{s}\biggr]\\
 & =\EE\biggl[\int_{s}^{t}\Bigl[f(x_{r}^{\pi})V_{r}+g(y_{r}^{\pi})W_{r}\Bigr]\,\D r\biggm|\Ff_{s}\biggr]+R_{t,s},
\end{align*}
where
\begin{gather*}
R_{t,s}\deq\sum_{k=1}^{n}\EE\biggl[\int_{t_{k}}^{t_{k+1}}\biggl\{\Bigl[f(x_{r}^{\pi})-f(x_{t_{k}}^{\pi})\Bigr]V_{r}+\Bigl[g(y_{t_{k}}^{\pi})-g(y_{r}^{\pi})\Bigr]W_{r}\biggr\}\,\D r\biggm|\Ff_{s}\biggr].
\end{gather*}
Estimating the above using H\"older's inequality, we have for the first
term
\begin{align*}
  \Biggl|\EE\biggl[\int_{t_{k}}^{t_{k+1}}&\Bigl[f(x_{r}^{\pi})-f(x_{t_{k}}^{\pi})\Bigr]V_{r}\,\D r\biggm|\Ff_{s}\biggr]\Biggr| \\
 & \leq\EE\biggl[\int_{t_{k}}^{t_{k+1}}\left|f(x_{r}^{\pi})-f(x_{t_{k}}^{\pi})\right|^{1+\frac{1}{\beta_{1}-1}}\,\D r\biggm|\Ff_{s}\biggr]^{\frac{\beta_{1}-1}{\beta_{1}}}\EE\biggl[\int_{t_{k}}^{t_{k+1}}\left|V_{r}\right|^{\beta_{1}}\,\D r\biggm|\Ff_{s}\biggr]^{\frac{1}{\beta_{1}}}\\
 & \leq\EE\biggl[\int_{t_{k}}^{t_{k+1}}\left|f(x_{r}^{\pi})-f(x_{t_{k}}^{\pi})\right|^{1+\frac{1}{\beta_{1}-1}}\,\D r\biggm|\Ff_{s}\biggr]^{\frac{\beta_{1}-1}{\beta_{1}}}\\
 & \qquad\times\EE\biggl[\int_{t_{k}}^{t_{k+1}}\left|V_{r}\right|^{\beta_{1}}\,\D r\biggm|\Ff_{s}\biggr]^{\frac{1}{\beta_{1}}}\quad\forall k\in\{1,\ldots,n-1\}.
\end{align*}
This is $o(|t_{k+1}-t_{k}|)$, since $x^{\pi}$ is c\`adl\`ag, $f,$ $g$,
and $g^{\prime}$ are continuous and bounded, and by Assumption~\ref{assu:main}
and Proposition~\ref{prop:regular-means-regular},
\begin{align*}
  \EE\biggl[\int_{t_{k}}^{t_{k+1}}&\left|V_{r}\right|^{\beta_{1}}\,\D r\biggm|\Ff_{s}\biggr]
  \\
 & \leq L_{c}^{\beta_{1}/\beta_{c}}\Vert g^{\prime}\Vert^{\beta_{1}}\EE\biggl[\int_{t_{k}}^{t_{k+1}}\left|\E^{-\alpha r}\int_{\AA}\left(1+\psi_{\XX}(x_{r}^{\pi})+\psi_{\AA}(a)\right)^{1/\beta_{c}}\pi_{r}(\D a)\right|^{\beta_{1}}\,\D r\biggm|\Ff_{s}\biggr]\\
 & \in\Oo(|t_{k+1}-t_{k}|).
\end{align*}
The second term in $R_{t,s}$ can be treated similarly, using the
continuity of $y^{\pi}$. Letting $\max_{k\in\{1,\ldots,n\}}|t_{k+1}-t_{k}|\to0$,
we have $R_{t,s}\to0$ and the claim follows.
\end{proof}
\begin{proof}[Proof of Proposition~\ref{prop:regular-forward-equation}]
Substituting $\phi_{n}$ into the forward equation, Eq.~(\ref{eq:KFE}),
rearranging, and by using the properties of $(\psi_{\UU},\psi_{\AA},\phi,L_{1},L_{\UU},L_{\AA},\beta_{1},\Lambda_{1},\text{\ensuremath{\Lambda}}_{\AA})$
given in Definition~\ref{def:regular}, we get for all $n\in\NN$
and $t\in\TT$,
\begin{align*}
\int_{\UU}\phi_{n}(u)\mu_{t}^{\UU}(\D u) & =\int_{\UU}\psi_{\UU}(u)\nu_{0}(\D u)+\int_{0}^{t}\int_{\UU\times\AA}A\phi_{n}(u,a,s)\mu_{s}(\D u\times\D a)\,\D s\\
 & \qquad+\int_{\UU}\left[\phi_{n}(u)-\psi_{\UU}(u)\right]\nu(\D u)\\
 & \leq\int_{\UU}\psi_{\UU}(u)\nu_{0}(\D u)+\int_{0}^{t}\int_{\UU\times\AA}A\phi_{n}(u,a,s)\mu_{s}(\D u\times\D a)\,\D s\\
 & \leq\int_{\UU}\psi_{\UU}(u)\nu_{0}(\D u)+\int_{0}^{t}\int_{\UU\times\AA}\Lambda_{1}\left(1+\psi_{\UU}(u)+\psi_{\AA}(a)\right)\mu_{s}(\D u\times\D a)\,\D s\\
 & \leq L_{\UU}+\int_{0}^{t}\int_{\UU}\Lambda_{1}\psi_{\UU}(u)\mu_{s}(\D u)\,\D s+\Lambda_{1}t+\Lambda_{1}L_{\AA}\frac{1}{\text{\ensuremath{\Lambda}}_{\AA}}\left(\E^{\text{\ensuremath{\Lambda}}_{\AA}t}-1\right).
\end{align*}
Applying the monotone convergence theorem, we have that 
\begin{gather*}
\int_{\UU}\psi_{\UU}(u)\mu_{t}^{\UU}(\D u)\leq L_{\UU}+\int_{0}^{t}\int_{\UU}\Lambda_{1}\psi_{\UU}(u)\mu_{s}(\D u)\,\D s+\Lambda_{1}t+\Lambda_{1}L_{\AA}\frac{1}{\text{\ensuremath{\Lambda}}_{\AA}}\left(\E^{\text{\ensuremath{\Lambda}}_{\AA}t}-1\right),
\end{gather*}
for all $t\in\TT$. By using Gr\"onwall's inequality,
\begin{align}
\int_{\UU}\psi_{\UU}(u)\mu_{t}^{\UU}(\D u) & \leq\left(L_{\UU}+\Lambda_{1}t+\Lambda_{1}L_{\AA}\frac{1}{\text{\ensuremath{\Lambda}}_{\AA}}\left(\E^{\text{\ensuremath{\Lambda}}_{\AA}t}-1\right)\right)\E^{\Lambda_{1}t},\label{eq:psi-gron}
\end{align}
and the coefficients $L_{\psi}$, $\Lambda_{\psi}$ satisfying Eq.~(\ref{eq:psi-gen-bound-1})
can be readily found.

Recalling the bound of $\psi$ in terms of $\psi_{\UU}$ and $\psi_{\AA}$,
see Definition~\ref{def:regular}(\emph{i}), it is clear that Eq.~(\ref{eq:psifin})
holds.

To show $\mu^{\UU}$ is continuous, we return to the forward equation,
and estimate for an arbitrary $f\in\dD(G)$,
\begin{align}
\biggl|\int_{\UU}f(u)\mu_{t}^{\UU}(\D u)-\int_{\UU}f(u)\mu_{s}^{\UU}(\D u)\biggr| & \leq\int_{s}^{t}\int_{\UU\times\AA}a_{f}\psi(u,a)\mu_{r}(\D u\times\D a)\,\D r\nonumber \\
 & \leq a_{f}(t-s)L_{1}\left(\frac{\E^{\Lambda_{\psi}t}-\E^{\Lambda_{\psi}s}}{t-s}\frac{L_{\psi}}{\Lambda_{\psi}}\right)^{\frac{1}{\beta_{1}}}.\label{eq:f-bound}
\end{align}
The right-hand side vanishes as $t\to s$, and since $\dD(G)$ is
convergence determining (see {e.g.} \cite[Theorem 3.4.5]{EK1986}),
$\mu$ is continuous. Finally, as the right-hand side of Eq.~(\ref{eq:f-bound})
is also independent of $\mu$ and $f$, it is clear that Eq.~(\ref{eq:psi-unicon})
holds.
\end{proof}
\begin{proof}[Proof of Lemma~\ref{lem:G-metric}]
We first show that the two given definitions are indeed equivalent.
We set
\begin{gather*}
d_{\Gg}^{\prime}\left(\mu,\nu\right)\deq\sup\biggl\{\left|\int f(u)\mu(\D u)-\int f(u)\nu(\D u)\right|\biggm|f\in\Gg,\,\Vert f\Vert+\Vert f\Vert_{\Gg}\leq1\biggr\}\quad\forall\mu,\nu\in\Pp(\UU),
\end{gather*}
and show $d_{\Gg}=d_{\Gg}^{\prime}$. This follows from the set $\Gg_{1}\deq\{f\in\Gg\mid\Vert f\Vert+\Vert f\Vert_{\Gg}\leq1\}$
being absorbing, see e.g. \cite[Proposition I.2]{Yosida1995}. Because
of this, for arbitrary $\mu,\nu\in\Pp(\UU)$,
\begin{align*}
d_{\Gg}\left(\mu,\nu\right) & =\sup\biggl\{\frac{\left|\int sf_{1}(u)\mu(\D u)-\int sf_{1}(u)\nu(\D u)\right|}{\Vert sf_{1}\Vert+\Vert sf_{1}\Vert_{\Gg}}\biggm|f_{1}\in\Gg_{1},\,s>0,\,f_{1}\neq0\biggr\}\\
 & =\sup\biggl\{\frac{\left|\int f_{1}(u)\mu(\D u)-\int f_{1}(u)\nu(\D u)\right|}{\Vert f_{1}\Vert+\Vert f_{1}\Vert_{\Gg}}\biggm|f_{1}\in\Gg_{1},\,f_{1}\neq0\biggr\}\\
 & \geq\sup\biggl\{\left|\int f_{1}(u)\mu(\D u)-\int f_{1}(u)\nu(\D u)\right|\biggm|f_{1}\in\Gg_{1}\biggr\}\\
 & =d_{\Gg}^{\prime}\left(\mu,\nu\right).
\end{align*}
On the other hand,
\begin{align*}
d_{\Gg}^{\prime}\left(\mu,\nu\right) & \geq\sup\biggl\{\left|\int f_{1}(u)\mu(\D u)-\int f_{1}(u)\nu(\D u)\right|\biggm|f_{1}\in\Gg,\,\Vert f_{1}\Vert+\Vert f_{1}\Vert_{\Gg}=1\biggr\}\\
 & =\sup\biggl\{\left|\int sf_{1}(u)\mu(\D u)-\int sf_{1}(u)\nu(\D u)\right|\biggm|f_{1}\in\Gg_{1},\,s>0,\,\Vert sf_{1}\Vert+\Vert sf_{1}\Vert_{\Gg}=1\biggr\}\\
 & =\sup\biggl\{\frac{\left|\int f_{1}(u)\mu(\D u)-\int f_{1}(u)\nu(\D u)\right|}{\Vert f_{1}\Vert+\Vert f_{1}\Vert_{\Gg}}\biggm|f_{1}\in\Gg_{1},\,f_{1}\neq0\biggr\}\\
 & =d_{\Gg}\left(\mu,\nu\right),
\end{align*}
and so $d_{\Gg}(\mu,\nu)=d_{\Gg}^{\prime}(\mu,\nu)$ for all $\mu,\nu\in\Pp(\UU)$.

Turning to the proof of part (\emph{i}), the mapping $d_{\Gg}$ is
clearly finite and symmetric, and satisfies the triangle inequality.
It is also apparent that if $\mu=\nu$, then $d_{\Gg}(\mu,\nu)=0$
for all $\mu,\nu\in\Gg$. Conversely, if $d_{\Gg}(\mu,\nu)=0$, then
\begin{gather*}
\left|\int f(u)\mu(\D u)-\int f(u)\nu(\D u)\right|=0\quad\forall f\in\Gg,
\end{gather*}
By \cite[Theorem 3.4.5(a)]{EK1986}, this implies that $\mu=\nu$.
Therefore, $d_{\Gg}$ is a metric. If $\{\mu_{n}\}_{n\in\NN}\cup\mu\subset\Pp(\UU)$
and $d(\mu_{n},\mu)\to0$ as $n\to\infty$, then \cite[Theorem 3.4.5(b)]{EK1986}
implies that $\mu_{n}\Rightarrow\mu$ as $n\to\infty$.

(\emph{ii}) First note that

\begin{align*}
\frac{\left|\int f(u)\mu(\D u)-\int f(u)\nu(\D u)\right|}{\Vert f\Vert+\Vert f\Vert_{\Gg}} & \geq\frac{\left|\int f(u)\mu(\D u)-\int f(u)\nu(\D u)\right|}{\Vert f\Vert+\Vert f\Vert_{\Gg^{\prime}}}\quad\forall f\in\Gg^{\prime}.
\end{align*}
Taking the supremum over $\Gg^{\prime}$, and using basic estimates,
we then get
\begin{align*}
\sup_{f\in\Gg}\left\{ \frac{\left|\int f(u)\mu(\D u)-\int f(u)\nu(\D u)\right|}{\Vert f\Vert+\Vert f\Vert_{\Gg}}\right\}  & \geq\sup_{f\in\Gg^{\prime}}\left\{ \frac{\left|\int f(u)\mu(\D u)-\int f(u)\nu(\D u)\right|}{\Vert f\Vert+\Vert f\Vert_{\Gg^{\prime}}}\right\} ,
\end{align*}
and the claim follows.
\end{proof}

\subsection{Other proofs}

The following proposition asserts some basic properties regarding
tightness of sets of measures. This is used frequently in weak convergence
arguments.
\begin{prop}
\label{prop:tightness}(\emph{i}) Let $\UU_{1}$ and $\UU_{2}$ be
topological spaces, $\{\mu_{n}\}_{n\in\NN}\subset\Mm(\UU_{1}\times\UU_{2})$,
and let $\mu_{n}^{1}\deq\mu_{n}^{\UU_{1}}=\mu_{n}(\cdot\times\UU_{2})$
and $\mu_{n}^{2}\deq\mu_{n}^{\UU_{2}}=\mu_{n}(\UU_{1}\times\cdot)$
be respectively the $\UU_{1}$ and $\UU_{2}$ marginals of $\mu_{n}$
for every \textbf{$n\in\NN$}. If $\{\mu_{n}^{1}\}_{n\in\NN}$ and
$\{\mu_{n}^{2}\}_{n\in\NN}$ are both tight, then $\{\mu_{n}\}_{n\in\NN}$
is tight. By extension, this statement holds for all finite Cartesian
products of topological spaces. (\emph{ii}) Let $\{\mu_{i}\}_{i\in I}\subset\Mm(\RR_{\geq0})$,
where $I$ is a (possibly uncountable) index set. Suppose $\phi$
is a non-decreasing non-negative measurable function and that there
is a $b>0$ such that $\int\phi(x)\mu_{i}(\D x)<b$ for all $i\in I\setminus F$
where $F$ is finite. Then $\{\mu_{i}\}_{i\in I}$ is tight. (\emph{iii})
Let $\UU$ be Polish, $\{\mu_{i}\}_{i\in I}\subset\Mm(\UU)$, and
$\phi:\UU\to\RR$ be inf-compact. If $\int\phi(x)\mu(\D x)<b$ for
all $i\in I\setminus F$ where $F$ is finite, then $\{\mu_{i}\}_{i\in I}$
is tight.
\end{prop}

\begin{proof}
(\emph{i}) By tightness of the marginals, for each $\epsilon>0$ we
can find compact $K_{\epsilon}^{1}\subset\UU_{1}$ and $K_{\epsilon}^{2}\subset\UU_{2}$
such that $\mu_{n}^{1}(K_{\epsilon}^{1\,\compl})<\epsilon/2$ and
$\mu_{n}^{2}(K_{\epsilon}^{2\,\compl})<\epsilon/2$ for all $n\in\NN$.
Let $K_{\epsilon}\deq K_{\epsilon}^{1}\times K_{\epsilon}^{2}$. As
a product of compact sets, Tychonoff's theorem \cite[Theorem 2.57]{Aliprantis2006}
states that $K_{\epsilon}$ is compact. Noting that $K_{\epsilon}^{\compl}\subset(K_{\epsilon}^{1\,\compl}\times\UU_{2})\cup(\UU_{1}\times K_{\epsilon}^{2\,\compl})$,
we have that $\mu_{n}(K_{\epsilon}^{\compl})\leq\mu_{n}(K_{\epsilon}^{1\,\compl}\times\UU_{2})+\mu_{n}(\UU_{1}\times K_{\epsilon}^{2\,\compl})=\mu_{n}^{1}(K_{\epsilon}^{1\,\compl})+\mu_{n}^{2}(K_{\epsilon}^{2\,\compl})<\epsilon$,
demonstrating that $\{\mu_{n}\}_{n\in\NN}$ is tight.

(\emph{iii}) It suffices to show that $\{\mu_{i}\}_{i\in I\setminus F}$
is tight; we may always add a finite collection of measures into it
and maintain tightness. Suppose this set is not tight. Then we can
find an $\epsilon>0$ such that for all compact $K\subset\UU$, there
is a measure $\mu_{i}$ for which $\mu_{i}(K^{\compl})\geq\epsilon$.
Let $K=\{u\in\UU\mid\phi(u)\leq b/\epsilon\}$, and select $\mu_{i}$
so that $\mu_{i}(K^{\compl})\geq\epsilon$. Then,
\begin{gather*}
\int\phi(u)\mu_{i}(\D u)=\int_{K}\phi(u)\mu_{i}(\D u)+\int_{K^{\compl}}\phi(u)\mu_{i}(\D u)\geq b,
\end{gather*}
a contradiction. Proof of (\emph{ii}) uses the same idea and is omitted.
\end{proof}

\bibliographystyle{plain}
\bibliography{DynAnalytic}

\end{document}